\documentclass[12pt]{amsart}
\usepackage{amsmath,amsthm,amsfonts,amssymb,mathrsfs}
\date{\today}

\input xymatrix
\xyoption{all}

\usepackage{color}

\usepackage{amssymb,cite}

\newtheorem{theorem}{Theorem}[section]
\newtheorem{proposition}[theorem]{Proposition}%[section]
\newtheorem{corollary}[theorem]{Corollary}%[section]
\newtheorem{lemma}[theorem]{Lemma}
\theoremstyle{definition}
\newtheorem{problem}[theorem]{Problem}%[section]
\newtheorem{remark}[theorem]{Remark}%[section]
\newtheorem{question}[theorem]{Question}

\usepackage{hyperref}

\begin{document}

\title[The monoid of monotone injective partial selfmaps of the poset $(\mathbb{N}^{3},\leqslant)$ ...]{The monoid of monotone injective partial selfmaps of the poset $(\mathbb{N}^{3},\leqslant)$ with cofinite domains and images}

\author[O.~Gutik and O.~Krokhmalna]{Oleg~Gutik and Olha Krokhmalna}
\address{Faculty of Mathematics, National University of Lviv, Universytetska 1, Lviv, 79000, Ukraine}
\email{oleg.gutik@lnu.edu.ua,  Olia709@i.ua}



\begin{abstract}
Let $n$ be a positive integer $\geqslant 2$ and $\mathbb{N}^n_{\leqslant}$ be the $n$-th power of positive integers with the product order of the usual order on $\mathbb{N}$. In the paper we study the semigroup of injective partial monotone selfmaps of $\mathbb{N}^n_{\leqslant}$ with cofinite domains and images. We show that the group of units $H(\mathbb{I})$ of the semigroup $\mathscr{P\!O}\!_{\infty}(\mathbb{N}^n_{\leqslant})$ is isomorphic to the group $\mathscr{S}_n$ of permutations of an $n$-element set, and describe the subsemigroup of idempotents of $\mathscr{P\!O}\!_{\infty}(\mathbb{N}^n_{\leqslant})$.  Also in the case $n=3$ we describe the property of elements of the semigroup $\mathscr{P\!O}\!_{\infty}(\mathbb{N}^3_{\leqslant})$ as partial bijections of the poset $\mathbb{N}^3_{\leqslant}$ and Green's relations on the semigroup $\mathscr{P\!O}\!_{\infty}(\mathbb{N}^3_{\leqslant})$. In particular we show that $\mathscr{D}=\mathscr{J}$ in $\mathscr{P\!O}\!_{\infty}(\mathbb{N}^3_{\leqslant})$.
\end{abstract}

\maketitle

%\tableofcontents

\section{Introduction and preliminaries}

We shall follow the terminology of~\cite{Clifford-Preston-1961-1967} and \cite{Howie-1995}.

In this paper we shall denote the cardinality of the set $A$ by $|A|$.  We shall identify all sets $X$ with their cardinality $|X|$. For an arbitrary positive integer $n$ by $\mathscr{S}_n$ we denote the group of permutations of an $n$-elements set. Also, for infinite subsets $A$ and $B$ of an infinite set $X$ we shall write $A{\subseteq^*}B$ if and only if there exists a finite subset $A_0$ of $A$ such that $A\setminus A_0\subseteq B$.

An algebraic semigroup $S$ is called {\it inverse} if for any element $x\in S$ there exists a unique $x^{-1}\in S$ such that $xx^{-1}x=x$ and $x^{-1}xx^{-1}=x^{-1}$. The element $x^{-1}$ is called the {\it inverse of} $x\in S$.

If $S$ is a semigroup, then we shall denote the subset of idempotents in $S$ by $E(S)$. If $S$ is an inverse semigroup then $E(S)$ is closed under multiplication and we shall refer to $E(S)$ as a \emph{band} (or the \emph{band of} $S$). If the band $E(S)$ is a non-empty subset of $S$ then the semigroup operation on $S$ determines the following partial order $\leqslant$ on $E(S)$: $e\leqslant f$ if and only if $ef=fe=e$. This order is called the {\em natural partial order} on $E(S)$. A \emph{semilattice} is a commutative semigroup of idempotents. A semilattice $E$ is called {\em linearly ordered} or a \emph{chain} if its natural order is a linear order.

If $S$ is a semigroup, then we shall denote Green's relations on $S$ by $\mathscr{R}$, $\mathscr{L}$, $\mathscr{J}$, $\mathscr{D}$ and $\mathscr{H}$ (see \cite{GreenJ-1951} or \cite[Section~2.1]{Clifford-Preston-1961-1967}):
\begin{align*}
    a\mathscr{R}b &\: \mbox{ if and only if } \: aS^1=bS^1;\\
    a\mathscr{L}b &\: \mbox{ if and only if } \: S^1a=S^1b;\\
    a\mathscr{J}b &\: \mbox{ if and only if } \: S^1aS^1=S^1bS^1;\\
    &\mathscr{D}=\mathscr{L}\circ\mathscr{R}=
          \mathscr{R}\circ\mathscr{L};\\
    &\qquad \mathscr{H}=\mathscr{L}\cap\mathscr{R}.
\end{align*}
The $\mathscr{R}$-class (resp., $\mathscr{L}$-, $\mathscr{H}$-, $\mathscr{D}$- or $\mathscr{J}$-class) of the semigroup $S$ which contains an element $a$ of $S$ will be denoted by $R_a$ (resp., $L_a$, $H_a$, $D_a$ or $J_a$).

If $\alpha\colon X\rightharpoonup Y$ is a partial map, then by $\operatorname{dom}\alpha$ and $\operatorname{ran}\alpha$ we denote the domain and the range of $\alpha$, respectively.

Let $\mathscr{I}_\lambda$ denote the set of all partial one-to-one transformations of an infinite set $X$ of cardinality $\lambda$ together with the following semigroup operation: $x(\alpha\beta)=(x\alpha)\beta$ if $x\in\operatorname{dom}(\alpha\beta)=\{ y\in\operatorname{dom}\alpha\colon y\alpha\in\operatorname{dom}\beta\}$,  for $\alpha,\beta\in\mathscr{I}_\lambda$. The semigroup $\mathscr{I}_\lambda$ is called the \emph{symmetric inverse semigroup} over the set $X$~(see \cite[Section~1.9]{Clifford-Preston-1961-1967}). The symmetric inverse semigroup was introduced by Wagner~\cite{Wagner-1952} and it plays a major role in the semigroup theory. An element $\alpha\in\mathscr{I}_\lambda$ is called \emph{cofinite}, if the sets $\lambda\setminus\operatorname{dom}\alpha$ and $\lambda\setminus\operatorname{ran}\alpha$ are finite.

If $X$ is a non-empty set and $\leqslant$ is a reflexive, antisymmetric, transitive binary relation on $X$ then $\leqslant$ is called a \emph{partial order} on $X$ and $(X,\leqslant)$ is said to be a \emph{partially ordered set} or shortly a \emph{poset}.

Let $(X,\leqslant)$ be a partially ordered set. A non-empty subset $A$ of $(X,\leqslant)$ is called:
\begin{itemize}
  \item a \emph{chain} if the induced partial order from $(X,\leqslant)$ onto $A$ is linear, i.e., any two elements from $A$ are comparable in $(X,\leqslant)$;
  \item an $\omega$-\emph{chain} if $A$ is order isomorphic to the set of negative integers with the usual order $\le$;
  \item an \emph{anti-chain} if any two distinct elements from $A$ are incomparable in $(X,\leqslant)$.
\end{itemize}
  For an arbitrary $x\in X$ and non-empty $A\subseteq X$ we denote
 \begin{equation*}
{\uparrow}x=\left\{y\in X\colon x\leqslant y\right\}, \quad {\downarrow}x=\left\{y\in X\colon y\leqslant x\right\}, \quad \uparrow A=\bigcup_{x\in A}{\uparrow}x \quad \hbox{and} \quad \downarrow A=\bigcup_{x\in A}{\downarrow}x.
\end{equation*}
We shall say that a partial map $\alpha\colon X\rightharpoonup X$ is \emph{monotone} if $x\leqslant y$ implies $(x)\alpha\leqslant(y)\alpha$ for $x,y\in \operatorname{dom}\alpha$.

Let $\mathbb{N}$ be the set of positive integers with the usual linear order $\le$ and $n\geqslant 2$ be an arbitrary positive integer. On the Cartesian power $\mathbb{N}^n=\underbrace{\mathbb{N}\times\cdots\times\mathbb{N}}_{n\hbox{\footnotesize{-times}}}$ we define the product partial order, i.e.,
\begin{equation*}
    (i_1,\ldots,i_n)\leqslant(j_1,\ldots,j_n) \qquad \hbox{if and only if} \qquad (i_k\leq j_k) \quad \hbox{for all} \quad k=1,\ldots,n.
\end{equation*}
Later the set $\mathbb{N}^n$ with this partial order will be denoted by $\mathbb{N}^n_{\leqslant}$.

For an arbitrary positive integer $n\geqslant 2$ by $\mathscr{P\!O}\!_{\infty}(\mathbb{N}^n_{\leqslant})$ we denote the \emph{semigroup of injective partial monotone selfmaps of $\mathbb{N}^n_{\leqslant}$ with cofinite domains and images.} Obviously, $\mathscr{P\!O}\!_{\infty}(\mathbb{N}^n_{\leqslant})$ is a submonoid of the semigroup $\mathscr{I}_\omega$ and $\mathscr{P\!O}\!_{\infty}(\mathbb{N}^n_{\leqslant})$ is a countable semigroup.

Furthermore, we shall denote the identity of the semigroup $\mathscr{P\!O}\!_{\infty}(\mathbb{N}^n_{\leqslant})$ by $\mathbb{I}$ and the group of units of
$\mathscr{P\!O}\!_{\infty}(\mathbb{N}^n_{\leqslant})$ by $H(\mathbb{I})$.

The \emph{bicyclic semigroup} (or the \emph{bicyclic monoid}) $\mathscr{C}(p,q)$ is the semigroup with the identity $1$ generated by two elements $p$ and $q$, subject only to the condition $pq=1$.
The bicyclic semigroup is bisimple and every one of its congruences is either trivial or a group congruence. Moreover, every homomorphism $h$ of the bicyclic semigroup is either an isomorphism or the image of ${\mathscr{C}}(p,q)$ under $h$ is a cyclic group~(see \cite[Corollary~1.32]{Clifford-Preston-1961-1967}). The bicyclic semigroup plays an important role in algebraic theory of semigroups and in the theory of topological semigroups. For example a well-known Andersen's result~\cite{Andersen-1952} states that a ($0$--)simple semigroup with an idempotent is completely ($0$--)simple if and only if it does not contain an isomorphic copy of the bicyclic semigroup. Semigroup topologizations and shift-continuous topologizations of generalizations  of the bicyclic monoid, they  embedding into compact-like topological semigroups was studied in  \cite{Bardyla-2016}--\cite{Bardyla-2019}, \cite{Bardyla-Gutik-2016, Bertman-West-1976, Chuchman-Gutik-2010, Eberhart-Selden-1969, Fihel-Gutik-2011}, \cite{Gutik-2015}--\cite{Gutik-Mokrytskyi-2020}, \cite{Gutik-Repovs-2011, Gutik-Repovs-2012, Hogan-1984, Mokrytskyi-2019} and \cite{Anderson-Hunter-Koch-1965, Banakh-Dimitrova-Gutik-2009, Banakh-Dimitrova-Gutik-2010, Bardyla-2020, Bardyla-Ravsky-2020,  Gutik-Repovs-2007, Hildebrant-Koch-1986}, respectively.

The bicyclic monoid is isomorphic to the semigroup of all bijections between upper-sets of the poset $(\mathbb{N},\le)$ (see: see Exercise IV.1.11(ii) in \cite{Petrich-1984}). So, the semigroup of injective isotone partial selfmaps with cofinite domains and images of positive integers is a generalization of the bicyclic semigroup. Hence, it is a natural problem to describe semigroups of injective isotone partial selfmaps with cofinite domains and images of posets with $\omega$-chain.

The semigroups $\mathscr{I}_{\infty}^{\!\nearrow}(\mathbb{N})$ and $\mathscr{I}_{\infty}^{\!\nearrow}(\mathbb{Z})$ of injective isotone partial selfmaps with cofinite domains and images of positive integers and integers, respectively, are studied in \cite{Gutik-Repovs-2011} and \cite{Gutik-Repovs-2012}. It was proved that the semigroups $\mathscr{I}_{\infty}^{\!\nearrow}(\mathbb{N})$ and $\mathscr{I}_{\infty}^{\!\nearrow}(\mathbb{Z})$ have similar properties to the bicyclic semigroup: they are bisimple and  every non-trivial homomorphic image $\mathscr{I}_{\infty}^{\!\nearrow}(\mathbb{N})$ and $\mathscr{I}_{\infty}^{\!\nearrow}(\mathbb{Z})$ is a group, and moreover the semigroup $\mathscr{I}_{\infty}^{\!\nearrow}(\mathbb{N})$ has $\mathbb{Z}(+)$ as a maximal group image and $\mathscr{I}_{\infty}^{\!\nearrow}(\mathbb{Z})$ has $\mathbb{Z}(+)\times\mathbb{Z}(+)$, respectively.

In the paper \cite{Gutik-Repovs-2015}  algebraic properties of the semigroup
$\mathscr{I}^{\mathrm{cf}}_\lambda$ of cofinite partial bijections of an infinite cardinal $\lambda$ are studied. It is shown that
$\mathscr{I}^{\mathrm{cf}}_\lambda$ is a bisimple inverse semigroup
and that for every non-empty chain $L$ in
$E(\mathscr{I}^{\mathrm{cf}}_\lambda)$ there exists an inverse
subsemigroup $S$ of $\mathscr{I}^{\mathrm{cf}}_\lambda$ such that
$S$ is isomorphic to the bicyclic semigroup and $L\subseteq E(S)$,
Green's relations on $\mathscr{I}^{\mathrm{cf}}_\lambda$ are described
and it is proved that every non-trivial congruence on
$\mathscr{I}^{\mathrm{cf}}_\lambda$ is a group congruence. Also, the structure of the quotient semigroup $\mathscr{I}^{\mathrm{cf}}_\lambda/\sigma$, where $\sigma$ is the least group congruence on $\mathscr{I}^{\mathrm{cf}}_\lambda$, is described.

In the paper \cite{Gutik-Pozdnyakova-2014} the semigroup
$\mathscr{I\!O}\!_{\infty}(\mathbb{Z}^n_{\operatorname{lex}})$ of monotone injective partial selfmaps of the set of $L_n\times_{\operatorname{lex}}\mathbb{Z}$ having cofinite domain and image, where $L_n\times_{\operatorname{lex}}\mathbb{Z}$ is the lexicographic product of $n$-elements chain and the set of integers with the usual linear order is studied. Green's relations on $\mathscr{I\!O}\!_{\infty}(\mathbb{Z}^n_{\operatorname{lex}})$ are described and it is shown that the semigroup $\mathscr{I\!O}\!_{\infty}(\mathbb{Z}^n_{\operatorname{lex}})$ is
bisimple and its projective congruences are established. Also, in \cite{Gutik-Pozdnyakova-2014}  it is  proved that $\mathscr{I\!O}\!_{\infty}(\mathbb{Z}^n_{\operatorname{lex}})$ is finitely generated, every automorphism of $\mathscr{I\!O}\!_{\infty}(\mathbb{Z})$ is inner, and it is shown that in the case $n\geqslant 2$ the semigroup $\mathscr{I\!O}\!_{\infty}(\mathbb{Z}^n_{\operatorname{lex}})$ has non-inner automorphisms. In \cite{Gutik-Pozdnyakova-2014} we proved that for every positive integer $n$ the quotient semigroup
$\mathscr{I\!O}\!_{\infty}(\mathbb{Z}^n_{\operatorname{lex}})/\sigma$, where $\sigma$ is a least group congruence on $\mathscr{I\!O}\!_{\infty}(\mathbb{Z}^n_{\operatorname{lex}})$, is isomorphic to the direct power $\left(\mathbb{Z}(+)\right)^{2n}$. The structure of the sublattice of congruences on $\mathscr{I\!O}\!_{\infty}(\mathbb{Z}^n_{\operatorname{lex}})$ which are contained in the least group congruence is described in \cite{Gutik-Pozdniakova-2014}.

In the paper \cite{Gutik-Pozdniakova-2016}  algebraic properties of the semigroup $\mathscr{P\!O}\!_{\infty}(\mathbb{N}^2_{\leqslant})$ are studied. The properties of elements of the semigroup $\mathscr{P\!O}\!_{\infty}(\mathbb{N}^2_{\leqslant})$ as monotone partial bijection of $\mathbb{N}^{2}_{\leqslant}$ are described and showed that the group of units of $\mathscr{P\!O}\!_{\infty}(\mathbb{N}^2_{\leqslant})$ is isomorphic to the cyclic group of order two. Also in \cite{Gutik-Pozdniakova-2016} the subsemigroup of idempotents of $\mathscr{P\!O}\!_{\infty}(\mathbb{N}^2_{\leqslant})$ and Green's relations on $\mathscr{P\!O}\!_{\infty}(\mathbb{N}^2_{\leqslant})$ are described. In particular, it is proved that $\mathscr{D}=\mathscr{J}$  in $\mathscr{P\!O}\!_{\infty}(\mathbb{N}^2_{\leqslant})$.
In \cite{Gutik-Pozdniakova-2016a} the natural partial order $\preccurlyeq$ on the semigroup $\mathscr{P\!O}\!_{\infty}(\mathbb{N}^2_{\leqslant})$ is described  and it is shown that it coincides with the natural partial order the induced from symmetric inverse monoid over the set $\mathbb{N}\times\mathbb{N}$ onto the semigroup $\mathscr{P\!O}\!_{\infty}(\mathbb{N}^2_{\leqslant})$. Also,  it is  proved that the semigroup $\mathscr{P\!O}\!_{\infty}(\mathbb{N}^2_{\leqslant})$ is isomorphic to the semidirect product $\mathscr{P\!O}\!_{\infty}^{\,+}(\mathbb{N}^2_{\leqslant})\rtimes \mathbb{Z}_2$ of the monoid $\mathscr{P\!O}\!_{\infty}^{\,+}(\mathbb{N}^2_{\leqslant})$ of orientation-preserving monotone injective partial selfmaps of $\mathbb{N}^{2}_{\leqslant}$ with cofinite domains and images by the cyclic group $\mathbb{Z}_2$ of order two. It is described the congruence $\sigma$ on the semigroup $\mathscr{P\!O}\!_{\infty}(\mathbb{N}^2_{\leqslant})$, which is generated by the natural order $\preccurlyeq$ on the semigroup $\mathscr{P\!O}\!_{\infty}(\mathbb{N}^2_{\leqslant})$: $\alpha\sigma\beta$ if and only if $\alpha$ and $\beta$ are comparable in $\left(\mathscr{P\!O}\!_{\infty}(\mathbb{N}^2_{\leqslant}),\preccurlyeq\right)$. It is proved that the quotient semigroup $\mathscr{P\!O}\!_{\infty}^{\,+}(\mathbb{N}^2_{\leqslant})/\sigma$ is isomorphic to the free commutative monoid $\mathfrak{AM}_\omega$ over an infinite countable set and it is shown that the quotient semigroup $\mathscr{P\!O}\!_{\infty}(\mathbb{N}^2_{\leqslant})/\sigma$ is isomorphic to the semidirect product of the free commutative monoid $\mathfrak{AM}_\omega$ by the group $\mathbb{Z}_2$.

In the paper \cite{Gutik-Savchuk-2018} the semigroup $\mathbf{I}\mathbb{N}_{\infty}$ of all partial co-finite isometries of positive integers  is studied. The semigroup $\mathbf{I}\mathbb{N}_{\infty}$ is  some generalization of the bicyclic monoid and it is a submonoid of $\mathscr{I}_{\infty}^{\!\nearrow}(\mathbb{N})$. Green's relations on the semigroup $\mathbf{I}\mathbb{N}_{\infty}$ and its band are described there and it is proved that $\mathbf{I}\mathbb{N}_{\infty}$ is a simple $E$-unitary $F$-inverse semigroup. Also there is described the least group congruence $\mathfrak{C}_{\mathbf{mg}}$ on $\mathbf{I}\mathbb{N}_{\infty}$ and it is proved that the quotient semigroup  $\mathbf{I}\mathbb{N}_{\infty}/\mathfrak{C}_{\mathbf{mg}}$ is isomorphic to the additive group of integers. An example of a non-group congruence on the semigroup $\mathbf{I}\mathbb{N}_{\infty}$ is presented. Also, it is proved that a congruence on the semigroup $\mathbf{I}\mathbb{N}_{\infty}$ is a group congruence if and only if its restriction onto an isomorphic  copy of the bicyclic semigroup in $\mathbf{I}\mathbb{N}_{\infty}$ is a group congruence.

In the paper \cite{Gutik-Savchuk-2019} submonoids of the monoid $\mathscr{I}_{\infty}^{\,\Rsh\!\!\!\nearrow}(\mathbb{N})$ of almost monotone injective co-finite partial selfmaps of positive integers $\mathbb{N}$ is established.
Let $\mathscr{C}_{\mathbb{N}}$ be the subsemigroup $\mathscr{I}_{\infty}^{\,\Rsh\!\!\!\nearrow}(\mathbb{N})$ which is generated by the partial shift $n\mapsto n+1$ and its inverse partial map.
In \cite{Gutik-Savchuk-2019} it was shown that every automorphism of a full inverse subsemigroup of $\mathscr{I}_{\infty}^{\!\nearrow}(\mathbb{N})$ which contains the semigroup $\mathscr{C}_{\mathbb{N}}$ is the identity map. Also there is constructed a submonoid $\mathbf{I}\mathbb{N}_{\infty}^{[\underline{1}]}$ of $\mathscr{I}_{\infty}^{\,\Rsh\!\!\!\nearrow}(\mathbb{N})$ with the following property: if $S$ is an inverse submonoid of $\mathscr{I}_{\infty}^{\,\Rsh\!\!\!\nearrow}(\mathbb{N})$ such that $S$ contains $\mathbf{I}\mathbb{N}_{\infty}^{[\underline{1}]}$ as a submonoid, then every non-identity congruence $\mathfrak{C}$ on $S$ is a group congruence.
Also,  it is proved that  if $S$ is an inverse submonoid of $\mathscr{I}_{\infty}^{\,\Rsh\!\!\!\nearrow}(\mathbb{N})$ such that $S$ contains $\mathscr{C}_{\mathbb{N}}$ as a submonoid then $S$ is simple and the quotient semigroup $S/\mathfrak{C}_{\mathbf{mg}}$, where $\mathfrak{C}_{\mathbf{mg}}$ is minimum group congruence on $S$, is isomorphic to the additive group of integers.

We observe that the semigroups  of all partial co-finite isometries of  integers are studied in \cite{Bezushchak-2004, Bezushchak-2008, Gutik-Savchuk-2017}.

The monoid $\mathbf{I}\mathbb{N}_{\infty}^n$ of cofinite partial isometries of the $n$-th power of the set of positive integers $\mathbb{N}$ with the usual metric for a positive integer $n\geqslant 2$ is studied in \cite{Gutik-Savchuk-2019-a}. The semigroup $\mathbf{I}\mathbb{N}_{\infty}^n$ is a submonoid of $\mathscr{P\!O}\!_{\infty}(\mathbb{N}^n_{\leqslant})$ for any positive integer $n\geqslant 2$. In \cite{Gutik-Savchuk-2019-a} it is proved that for any integer $n\geqslant 2$ the semigroup $\mathbf{I}\mathbb{N}_{\infty}^n$ is isomorphic to the semidirect product ${\mathscr{S}_n\ltimes_\mathfrak{h}(\mathscr{P}_{\infty}(\mathbb{N}^n),\cup)}$ of the free semilattice with the unit $(\mathscr{P}_{\infty}(\mathbb{N}^n),\cup)$ by the symmetric group $\mathscr{S}_n$.

Later in this paper we shall assume that $n$ is an arbitrary positive integer $\geqslant 2$.

In this paper we study the semigroup of injective partial monotone selfmaps of the poset $\mathbb{N}^n_{\leqslant}$ with cofinite domains and images. We show that the group of units $H(\mathbb{I})$ of the monoid $\mathscr{P\!O}\!_{\infty}(\mathbb{N}^n_{\leqslant})$ is isomorphic to the group $\mathscr{S}_n$ and describe the subgroup of idempotents of $\mathscr{P\!O}\!_{\infty}(\mathbb{N}^n_{\leqslant})$.  Also in the case $n=3$ we describe the property of elements of the semigroup $\mathscr{P\!O}\!_{\infty}(\mathbb{N}^n_{\leqslant})$ as partial bijections of the poset $\mathbb{N}^n_{\leqslant}$ and Green's relations on the semigroup $\mathscr{P\!O}\!_{\infty}(\mathbb{N}^3_{\leqslant})$. In particular we show that $\mathscr{D}=\mathscr{J}$ in $\mathscr{P\!O}\!_{\infty}(\mathbb{N}^3_{\leqslant})$.

%%%%%%%%%%%%%%%%%%%%%%%%%%%%%%%%%%%%%%%%%%%%%%%%%%%%%%%%%%%%%%%%%%

\section{Properties of elements of the semigroup
$\mathscr{P\!O}\!_{\infty}(\mathbb{N}^n_{\leqslant})$ as monotone partial permutations}\label{section-2}

In this short section we describe properties of elements of the semigroup
$\mathscr{P\!O}\!_{\infty}(\mathbb{N}^n_{\leqslant})$ as monotone partial transformations of the poset $\mathbb{N}^n_{\leqslant}$.

It is obvious that the group of units $H(\mathbb{I})$ of the semigroup $\mathscr{P\!O}\!_{\infty}(\mathbb{N}^n_{\leqslant})$ consists of exactly all order isomorphisms of the poset $\mathbb{N}^n_{\leqslant}$ and hence Theorem~2.8 of \cite{Gutik-Mokrytskyi-2020} implies  the following

\begin{theorem}\label{theorem-2.1}
For any positive integer $n$ the group of units $H(\mathbb{I})$ of the semigroup $\mathscr{P\!O}\!_{\infty}(\mathbb{N}^n_{\leqslant})$ is isomorphic to the group $\mathscr{S}_n$ of permutations of an $n$-elements set. Moreover, every element of $H(\mathbb{I})$ permutates coordinates of elements of $\mathbb{N}^n$, and only these permutations are elements of $H(\mathbb{I})$.
\end{theorem}

Since every $\alpha\in \mathscr{P\!O}\!_{\infty}(\mathbb{N}^n_{\leqslant})$ is a cofinite monotone partial transformation of the poset $\mathbb{N}^n_{\leqslant}$ the following statement holds.

\begin{lemma}\label{lemma-2.2}
If $(1.\ldots,1)\in\operatorname{dom}\alpha$ for some $\alpha\in \mathscr{P\!O}\!_{\infty}(\mathbb{N}^n_{\leqslant})$ then  $(1.\ldots,1)\alpha=(1.\ldots,1)$.
\end{lemma}

For an arbitrary $i=1,\ldots,n$ define
\begin{equation*}
\mathscr{K}_i=\big\{(1,\ldots,\underbrace{m}_{i\hbox{\footnotesize{th}}},\ldots,1)\in\mathbb{N}^n\colon m\in\mathbb{N}\big\}
\end{equation*}
and by $\mathfrak{pr}_i\colon\mathbb{N}^n\to \mathbb{N}^n$ denote the projection onto the $i$-th coordinate, i.e., for every $(m_1,\ldots,m_i,\ldots,m_n)\in\mathbb{N}^n$ put
\begin{equation*}
(m_1,\ldots,\underbrace{m_i}_{i\hbox{\footnotesize{th}}},\ldots,m_n)\mathfrak{pr}_i= (1,\ldots,\underbrace{m_i}_{i\hbox{\footnotesize{th}}},\ldots,1).
\end{equation*}

\begin{lemma}\label{lemma-2.3}
Let $\left\{\overline{x}_1,\ldots,\overline{x}_k\right\}$ be a set of points in $\mathbb{N}^n\setminus\{(1,\ldots,1)\}$, $k\in\mathbb{N}$. Then the set $\mathbb{N}^n\setminus\left({\uparrow}\overline{x}_1\cup\ldots\cup{\uparrow}\overline{x}_k\right)$ is finite if and only if $k\geqslant n$ and for every $\mathscr{K}_i$, $i=1,\ldots,n$, there exists $\overline{x}_j\in \left\{\overline{x}_1,\ldots,\overline{x}_k\right\}$ such that $\overline{x}_j\in \mathscr{K}_i$.
\end{lemma}

\begin{proof}[\textsl{Proof}]
$(\Leftarrow)$ Without loss of generality we may assume that $\overline{x}_j\in \mathscr{K}_j$ for every positive integer $j\leqslant n$. Then simple verifications imply that the set $\mathbb{N}^n\setminus\left({\uparrow}\overline{x}_1\cup\ldots\cup{\uparrow}\overline{x}_n\right)$ is finite, and hence so is the set $\mathbb{N}^n\setminus\left({\uparrow}\overline{x}_1\cup\ldots\cup{\uparrow}\overline{x}_k\right)$.

$(\Rightarrow)$
Suppose to the contrary that there exist a subset $\left\{\overline{x}_1,\ldots,\overline{x}_k\right\}\subseteq \mathbb{N}^n\setminus\{(1,\ldots,1)\}$ and an integer $i\in \{1,\ldots,n\}$ such that $\mathbb{N}^n\setminus\left({\uparrow}\overline{x}_1\cup\ldots\cup{\uparrow}\overline{x}_k\right)$ is finite and $\overline{x}_j\notin \mathscr{K}_i$ for any $j\in\{1,\ldots,k\}$.

The definition of $\mathscr{K}_i$ ($i=1,\ldots,n$) implies that $\mathscr{K}_i$ with the induced partial order from $\mathbb{N}^n_{\leqslant}$ is an $\omega$-chain such that ${\downarrow}\mathscr{K}_i=\mathscr{K}_i$. Hence, for any $\overline{x}\in \mathbb{N}^n$ we have that either $\mathscr{K}_i\setminus {\uparrow}\overline{x}$ is finite or $\mathscr{K}_i\cap{\uparrow}\overline{x}=\varnothing$. Then by our assumption we get that the set $\mathbb{N}^n\setminus\left({\uparrow}\overline{x}_1\cup\ldots\cup{\uparrow}\overline{x}_n\right)$ is infinite, a contradiction. The inequality $k\geqslant n$  follows from the above arguments.
\end{proof}

Later for an arbitrary non-empty subset $A$ of $\mathbb{N}^n$ by $\varepsilon_A$ we shall denote the identity map of the set $\mathbb{N}^n\setminus A$. It is obvious that the following lemma holds.

\begin{lemma}\label{lemma-2.4}
For an arbitrary non-empty subset $A$ of $\mathbb{N}^n$, $\varepsilon_A$ is an element of the semigroup $\mathscr{P\!O}\!_{\infty}(\mathbb{N}^n_{\leqslant})$, and hence so are $\varepsilon_A\alpha$, $\alpha\varepsilon_A$, and $\varepsilon_A\alpha\varepsilon_A$ for any $\alpha\in\mathscr{P\!O}\!_{\infty}(\mathbb{N}^n_{\leqslant})$.
\end{lemma}

\begin{proposition}\label{proposition-2.5}
For an arbitrary element $\alpha$ of the semigroup $\mathscr{P\!O}\!_{\infty}(\mathbb{N}^n_{\leqslant})$ there exists a unique permutation $\mathfrak{s}\colon \{1,\ldots,n\}\to\{1,\ldots,n\}$ such that $(\mathscr{K}_i\cap\operatorname{dom}\alpha)\alpha\subseteq \mathscr{K}_{(i)\mathfrak{s}}$ for any $i=1,\ldots,n$.
\end{proposition}

\begin{proof}[\textsl{Proof}]
Lemma~\ref{lemma-2.4} implies that without loss of generality we may assume that $(1,\ldots,1)\notin \operatorname{dom}\alpha$ and $(1,\ldots,1)\notin \operatorname{ran}\alpha$.

Since for any $i=1,\ldots,n$ the set $\mathscr{K}_i$ with the induced order from the poset $\mathbb{N}^n_{\leqslant}$ is an $\omega$-chain, the set $\mathscr{K}_i\cap\operatorname{dom}\alpha$ contains the least element $\overline{l}_i^\alpha$. By Lemma~\ref{lemma-2.3} the set $\mathbb{N}^n\setminus\big({\uparrow}\overline{l}_1^\alpha\cup\cdots\cup{\uparrow}\overline{l}_n^\alpha\big)$ is finite and hence so is $\operatorname{dom}\alpha\setminus\big({\uparrow}\overline{l}_1^\alpha\cup\cdots\cup{\uparrow}\overline{l}_n^\alpha\big)$. Since $\alpha$ is a cofinite partial bijection of $\mathbb{N}^n$, we have that
\begin{equation*}
\big({\uparrow}\overline{l}_1^\alpha\cup\cdots\cup{\uparrow}\overline{l}_n^\alpha\big)\alpha= \big({\uparrow}\overline{l}_1^\alpha\big)\alpha\cup\cdots\cup\big({\uparrow}\overline{l}_n^\alpha\big)\alpha
\end{equation*}
and the set $\mathbb{N}^n\setminus \big(\big({\uparrow}\overline{l}_1^\alpha\big)\alpha\cup\cdots\cup\big({\uparrow}\overline{l}_n^\alpha\big)\alpha\big)$ is finite. Also, since $\alpha$ is a monotone partial bijection of the poset $\mathbb{N}^n_{\leqslant}$ we obtain that $\big({\uparrow}\overline{l}_i^\alpha\big)\alpha\subseteq {\uparrow}\big(\overline{l}_i^\alpha\big)\alpha$ for all $i=1,\ldots,n$. Then by Lemma~\ref{lemma-2.3} there exists a permutation $\mathfrak{s}\colon \{1,\ldots,n\}\to\{1,\ldots,n\}$ such that $(\overline{l}_i^\alpha)\alpha\in \mathscr{K}_{(i)\mathfrak{s}}$ for any $i=1,\ldots,n$, because
\begin{equation*}
\mathbb{N}^n\setminus \big({\uparrow}\big(\overline{l}_1^\alpha\big)\alpha\cup\cdots\cup\big({\uparrow}\overline{l}_n^\alpha\big)\alpha\big)\subseteq \mathbb{N}^n\setminus \big(\big({\uparrow}\overline{l}_1^\alpha\big)\alpha\cup\cdots\cup\big({\uparrow}\overline{l}_n^\alpha\big)\alpha\big)
\end{equation*}
and the set $\mathbb{N}^n\setminus \big({\uparrow}\big(\overline{l}_1^\alpha\big)\alpha\cup\cdots\cup\big({\uparrow}\overline{l}_n^\alpha\big)\alpha\big)$ is finite.
This implies that $(\overline{x})\alpha\in \mathscr{K}_{(i)\mathfrak{s}}$ for all $\overline{x}\in\mathscr{K}_i\cap\operatorname{dom}\alpha$ and  any $i=1,\ldots,n$.

The proof of uniqueness of the permutation $\mathfrak{s}$ for $\alpha\in\mathscr{P\!O}\!_{\infty}(\mathbb{N}^n_{\leqslant})$ is trivial. This completes the proof of the proposition.
\end{proof}

Theorem~\ref{theorem-2.1} and Proposition~\ref{proposition-2.5} imply the following corollary.

\begin{corollary}\label{corollary-2.6}
For every element $\alpha$ of the semigroup $\mathscr{P\!O}\!_{\infty}(\mathbb{N}^n_{\leqslant})$ there exists a unique element $\sigma$ of the group of units $H(\mathbb{I})$ of $\mathscr{P\!O}\!_{\infty}(\mathbb{N}^n_{\leqslant})$ such that $(\mathscr{K}_i\cap\operatorname{dom}\alpha)\alpha\sigma\subseteq \mathscr{K}_i$ and $(\mathscr{K}_i\cap \operatorname{dom}\alpha)\sigma^{-1}\alpha\subseteq \mathscr{K}_i$ for all $i=1,\ldots,n$.
\end{corollary}

\begin{lemma}\label{lemma-2.7}
There is no a finite family $\{L_1,\ldots,L_k\}$ of chains in the poset $\mathbb{N}^2_{\leqslant}$ such that $\mathbb{N}^2=L_1\cup\cdots\cup L_k$. Moreover, every co-finite subset in $\mathbb{N}^2_{\leqslant}$ has this property.
\end{lemma}

\begin{proof}[\textsl{Proof}]
Suppose to the contrary that there exists a positive integer $k$ such that $\mathbb{N}^2=L_1\cup\cdots\cup L_k$ and $L_i$ is a chain for each $i=1,\ldots,k$. Then
\begin{equation*}
\left\{(1,k+1),(2,k),\ldots,(k,2),(k+1,1)\right\}
\end{equation*}
is an anti-chain in the poset $\mathbb{N}^2_{\leqslant}$ which contains exactly $k+1$ elements.
Without loss of generality we may assume that $L_i\cap L_j=\varnothing$ for $i\neq j$.
Since $\mathbb{N}^2=L_1\sqcup\cdots\sqcup L_k$,  by the pigeonhole principle (or by the Dirichlet drawer principle, see \cite[Section~7.3]{Berman-Fryer-1972}) there exists a chain $L_i$, $i=1,\ldots,k$, which contains at least two distinct elements of the set $\left\{(1,k+1),(2,k),\ldots,(k,2),(k+1,1)\right\}$, a contradiction.

Assume that $A$ is a co-finite subset of $\mathbb{N}^2_{\leqslant}$ such that $A=\mathbb{N}^2\setminus\left\{x_1,\ldots,x_p\right\}$ for some positive integer $p$. For every $i=1,\ldots,p$ we put $L_{k+i}= \{x_i\}$. Then for every finite partition $\{L_1,\ldots,L_k\}$ of $A$ such that $L_i$ is a chain for each $i=1,\ldots,k$ the  family $\{L_1,\ldots,L_k,L_{k+1}\ldots,L_{k+p}\}$ is a finite partition of the poset $\mathbb{N}^2_{\leqslant}$ such that $L_i$ is a chain for each $i=1,\ldots,k+p$. This contradicts the above part of the proof, and hence the second statement of the lemma holds.
\end{proof}

For any distinct $i,j\in\{1,\ldots,n\}$ we denote
\begin{equation*}
  \mathscr{K}_{i,j}=\left\{(x_1,\ldots,x_n)\in\mathbb{N}^n\colon x_k=1 \; \hbox{~for all~} \; k\in \{1,\ldots,n\}\setminus\{i,j\}\right\}
\end{equation*}
and
\begin{equation*}
  \mathscr{K}_{i,j}^{\circ}=\mathscr{K}_{i,j}\setminus\left(\mathscr{K}_{i}\cup\mathscr{K}_{j}\right)
\end{equation*}

\begin{lemma}\label{lemma-2.8}
Let $n$ be a positive integer $\geqslant 3$. Let $\overline{x}_i$ be an arbitrary element of $\mathscr{K}_{i}\setminus\left\{1,\ldots,1\right\}$ for $i=3,\ldots, n$ and $\overline{y}_{1,2}$ be an arbitrary element of $\mathscr{K}_{1,2}^{\circ}$. Then there exists a finite family $\{L_1,\ldots,L_k\}$ of chains in the poset $\mathbb{N}^n_{\leqslant}$ such that
\begin{equation*}
L_1\cup\cdots\cup L_k= \mathbb{N}^n\setminus\left({\uparrow}\overline{y}_{1,2}\cup{\uparrow}\overline{x}_3\cup\cdots\cup{\uparrow}\overline{x}_n\right).
\end{equation*}
\end{lemma}

\begin{proof}[\textsl{Proof}]
Let $\overline{x}_i=(1,1,\ldots,\underbrace{x_i}_{i\hbox{\footnotesize{th}}},\ldots,1)$ for $i=3,\ldots,n$ and $\overline{y}_{1,2}=(y_1,y_2,1\ldots,1)$. Then for any element $\overline{a}=\left(a_1,\ldots,a_n\right)$ of the set $\mathbb{N}^n\setminus\left({\uparrow}\overline{y}_{1,2}\cup{\uparrow}\overline{x}_3\cup\cdots\cup{\uparrow}\overline{x}_n\right)$ the following conditions hold:
\begin{itemize}
  \item[$(i)$] $a_i<x_i$ for any $i=3,\ldots,n$;
  \item[$(ii)$] if $a_1\geqslant y_1$ then $a_2<y_2$;
  \item[$(iii)$] if $a_2\geqslant y_2$ then $a_1<y_1$.
\end{itemize}
These conditions imply that
\begin{equation*}
  \mathbb{N}^n\setminus\left({\uparrow}\overline{y}_{1,2}\cup{\uparrow}\overline{x}_3\cup\cdots\cup{\uparrow}\overline{x}_n\right)= \bigcup\left\{S(k_3,\ldots,k_n)\colon k_3<x_3,\ldots, k_n<x_n\right\},
\end{equation*}
where
\begin{equation*}
\begin{split}
  S(k_3,\ldots,k_n)& =\bigcup\left\{L_i(k_3,\ldots,k_n)\colon i=1,\ldots y_1-1\right\}\cup \\
    & \qquad \cup \bigcup\left\{R_j(k_3,\ldots,k_n)\colon j=1,\ldots y_2-1\right\},
\end{split}
\end{equation*}
with
\begin{equation*}
  L_i(k_3,\ldots,k_n)=\left\{(i,p,k_3,\ldots,k_n)\in\mathbb{N}^n\colon p\in\mathbb{N}\right\}
\end{equation*}
and
\begin{equation*}
 R_j(k_3,\ldots,k_n)=\left\{(p,j,k_3,\ldots,k_n)\in\mathbb{N}^n\colon p\in\mathbb{N}\right\}.
\end{equation*}
We observe that for arbitrary positive integers $i$,$j$, $k_3,\ldots,k_n$ the sets $L_i(k_3,\ldots,k_n)$ and $R_j(k_3,\ldots,k_n)$ are chains in the poset $\mathbb{N}^n_{\leqslant}$. Since the set $\mathbb{N}^n\setminus\left({\uparrow}\overline{y}_{1,2}\cup{\uparrow}\overline{x}_3\cup\cdots\cup{\uparrow}\overline{x}_n\right)$ is the union of finitely many sets of the form $S(k_3,\ldots,k_n)$ the above arguments imply the required  statement of the lemma.
\end{proof}

\begin{proposition}\label{proposition-2.9}
Let $\alpha$ be an element of $\mathscr{P\!O}\!_{\infty}(\mathbb{N}^n_{\leqslant})$ such that $(\mathscr{K}_i\cap\operatorname{dom}\alpha)\alpha\subseteq \mathscr{K}_i$ for all $i=1,\ldots,n$. Then $(\mathscr{K}_{i_1,i_2}\cap\operatorname{dom}\alpha)\alpha\subseteq \mathscr{K}_{i_1,i_2}$ for all distinct $i_1,i_2=1,\ldots,n$.
\end{proposition}

\begin{proof}[\textsl{Proof}]
Suppose to the contrary that there exists $\overline{x}\in \mathscr{K}_{i_1,i_2}\cap\operatorname{dom}\alpha$ such that $(\overline{x})\alpha\notin\mathscr{K}_{i_1,i_2}$. By Theorem~\ref{theorem-2.1} without loss of generality we may assume that $i_1=1$ and $i_2=2$, i.e., $\overline{x}\in \mathscr{K}_{1,2}$ and $(\overline{x})\alpha\notin\mathscr{K}_{1,2}$. By Lemma~\ref{lemma-2.2}, $\overline{x}\neq(1,\ldots,1)$.

For every $i=3,\ldots,n$ we let $\overline{x}_i^\alpha=(1,1,\ldots,\underbrace{x_i^\alpha}_{i\hbox{\footnotesize{th}}},\ldots,1)\in\operatorname{dom}\alpha$ be the smallest element of $\mathscr{K}_{i}$ such that $(\overline{x}_i^\alpha)\alpha\neq(1,\ldots,1)$. There exists  $x_{1,2}^\alpha=(x_1^\alpha,x_2^\alpha,1,\ldots,1)\in\operatorname{dom}\alpha\cap \mathscr{K}_{1,2}^\circ$ such that  $\overline{x}\leqslant \overline{x}_{1,2}^\alpha$. Since $\alpha\in\mathscr{P\!O}\!_{\infty}(\mathbb{N}^n_{\leqslant})$,  $(\overline{x})\alpha\leqslant (\overline{x}_{1,2}^\alpha)\alpha\notin\mathscr{K}_{1,2}$.

Now, the monotonicity of $\alpha$ implies that $\left({\uparrow}\overline{x}_{1,2}^\alpha\right)\alpha\subseteq {\uparrow}\left(\overline{x}_{1,2}^\alpha\right)\alpha$ and $\left({\uparrow}\overline{x}_i^\alpha\right)\alpha\subseteq {\uparrow}\left(\overline{x}_i^\alpha\right)\alpha$ for any $i=3,\ldots,n$. By our assumption we have that
\begin{equation*}
  \mathscr{K}_{1,2}\cap\operatorname{ran}\alpha\subseteq \left(\mathbb{N}^n_{\leqslant}\setminus\left({\uparrow}\overline{x}_{1,2}^\alpha\cup{\uparrow}\overline{x}_3^\alpha\cup\cdots \cup{\uparrow}\overline{x}_n^\alpha\right)\right)\alpha.
\end{equation*}
Since the partial transformation $\alpha$ preserves chains in the poset $\mathbb{N}^n_{\leqslant}$, Lemma~\ref{lemma-2.8} implies that the set $\mathscr{K}_{1,2}\cap\operatorname{ran}\alpha$ is a union of finitely many chains, which contradicts Lemma~\ref{lemma-2.7}. The obtained contradiction implies the assertion of the proposition.
\end{proof}

\begin{theorem}\label{theorem-2.10}
Let $\alpha$ be an element of the semigroup $\mathscr{P\!O}\!_{\infty}(\mathbb{N}^3_{\leqslant})$ such that $(\mathscr{K}_i\cap\operatorname{dom}\alpha)\alpha\subseteq \mathscr{K}_i$ for all $i=1,2,3$. Then the following assertions hold:
\begin{itemize}
  \item[$(i)$] if $(x_1,x_2,x_3)\in\operatorname{dom}\alpha$ and $(x_1,x_2,x_3)\alpha=(x_1^\alpha,x_2^\alpha,x_3^\alpha)$ then $x_1^\alpha\le x_1$, $x_2^\alpha\le x_2$ and $x_3^\alpha\le x_3$ and hence $(\overline{x})\alpha\leqslant\overline{x}$ for any $\overline{x}\in\operatorname{dom}\alpha$;
  \item[$(ii)$] there exists a smallest positive integer $n_{\alpha}$ such that $(x_1,x_2,x_3)\alpha=(x_1,x_2,x_3)$ for all $(x_1,x_2,x_3)\in\operatorname{dom}\alpha\cap{\uparrow}(n_{\alpha},n_{\alpha},n_{\alpha})$.
\end{itemize}
\end{theorem}

\begin{proof}[\textsl{Proof}]
$(i)$ We shall prove the inequality $x_1^\alpha\le x_1$ by induction. The proofs of the inequalities $x_2^\alpha\le x_2$ and $x_3^\alpha\le x_3$ are similar.

By Proposition~\ref{proposition-2.9} we have that if $x_1=1$ then $x_1^\alpha=1$, as well.

\smallskip

Next we shall show that the following statement holds:
\begin{itemize}
  \item[] \emph{if for some positive integer $p>1$ the inequality $x_1<p$ implies $x_1^\alpha\le x_1$ then the equality $x_1=p$ implies $x_1^\alpha\le x_1$, too.}
\end{itemize}

Suppose to the contrary that there exists $(x_1,x_2,x_3)\in\operatorname{dom}\alpha$ such that
\begin{equation*}
x_1=p=(x_1,x_2,x_3)\mathfrak{pr}_1, \quad  (x_1,x_2,x_3)\alpha=(x_1^\alpha,x_2^\alpha,x_3^\alpha) \quad \hbox{and} \quad x_1+1\leq x_1^\alpha.
\end{equation*}
We define a partial map $\varpi\colon \mathbb{N}^3\rightharpoonup \mathbb{N}^3$ with $\operatorname{dom}\varpi=\mathbb{N}^3\setminus\left(\{1\}\times L(x_2)\times L(x_2)\right)$ and $\operatorname{ran}\varpi=\mathbb{N}^3$ by the formula
\begin{equation*}
  (i_1,i_2,i_3)\varpi=
  \left\{
    \begin{array}{ll}
      (i_1-1,i_2,i_3), & \hbox{if~} i_2\in L(x_2) \hbox{~and~} i_3\in L(x_2);\\
      (i_1,i_2,i_3),   & \hbox{otherwise,}
    \end{array}
  \right.
\end{equation*}
where $L(x_2)=\left\{1,\ldots,x_2\right\}$ and $L(x_3)=\left\{1,\ldots,x_3\right\}$. It is obvious that $\varpi\in\mathscr{P\!O}\!_{\infty}(\mathbb{N}^3_{\leqslant})$, and hence $\gamma\varpi^k\in\mathscr{P\!O}\!_{\infty}(\mathbb{N}^3_{\leqslant})$ for any positive integer $k$ and any $\gamma\in\mathscr{P\!O}\!_{\infty}(\mathbb{N}^3_{\leqslant})$.  This observation implies that without loss of generality we may assume that $x_1^\alpha=x_1+1$. Then the assumption of the theorem implies that  there exists the smallest element $(i_{\mathrm{m}},1,1)$ of $\mathscr{K}_1$ such that $i_{\mathrm{m}}^\alpha>x_1^\alpha+1$, where $(i_{\mathrm{m}}^\alpha,1,1)=(i_{\mathrm{m}},1,1)\alpha$. Since $({\uparrow}(i_{\mathrm{m}},1,1))\alpha\subseteq {\uparrow}(i_{\mathrm{m}}^\alpha,1,1)$, $({\uparrow}(x_1,x_2,x_3))\alpha\subseteq {\uparrow}(x_1^\alpha,x_2^\alpha,x_3^\alpha)$ and the set
$\mathbb{N}^3\setminus\operatorname{ran}\alpha$ is finite, our assumption implies that the set
\begin{equation*}
\mathscr{S}_{x_1}(\alpha)=\left\{\left(x_1,p_2,p_3\right)\in\operatorname{dom}\alpha\colon p_2,p_3\in\mathbb{N}\right\}
\end{equation*}
is a union of finitely many subchains of the poset $(\mathbb{N}^3,\leq)$. This contradicts Lemma~\ref{lemma-2.7} because the set $\mathscr{S}_{x_1}(\alpha)$ with the induced partial order from $\mathbb{N}^3_{\leqslant}$ is order isomorphic to a  cofinite subset of the poset $\mathbb{N}^2_{\leqslant}$. The obtained contradiction implies the requested inequality $x_1^\alpha\le x_1$ and hence we have that statement $(i)$ holds.

The last assertion of $(i)$ follows from the definition of the poset $\mathbb{N}^3_{\leqslant}$.

\smallskip

$(ii)$ Fix an arbitrary $\alpha\in\mathscr{P\!O}\!_{\infty}(\mathbb{N}^3_{\leqslant})$ such that $(\mathscr{K}_i\cap\operatorname{dom}\alpha)\alpha\subseteq \mathscr{K}_i$ for all $i=1,2,3$. Suppose to the contrary that for any positive integer $n$ there exists
\begin{equation*}
(x_1,x_2,x_3)\in\operatorname{dom}\alpha\cap{\uparrow}(n,n,n)
\end{equation*}
such that $(x_1,x_2,x_3)\alpha\neq(x_1,x_2,x_3)$.  We put $\textsf{N}_{\operatorname{dom}\alpha}=\left|\mathbb{N}^3\setminus\operatorname{dom}\alpha\right|+1$ and
\begin{equation*}
\begin{split}
  \textsf{M}_{\operatorname{dom}\alpha}& =\max\big\{\left\{x_1\colon(x_1,x_2,x_3)\notin\operatorname{dom}\alpha\right\}, \left\{x_2\colon(x_1,x_2,x_3)\notin\operatorname{dom}\alpha\right\},  \\
    & \qquad \left\{x_3\colon(x_1,x_2,x_3)\notin\operatorname{dom}\alpha\right\}\big\}+1.
\end{split}
\end{equation*}
The definition of the semigroup $\mathscr{P\!O}\!_{\infty}(\mathbb{N}^3_{\leqslant})$ implies that the positive integers $\textsf{N}_{\operatorname{dom}\alpha}$ and $\textsf{M}_{\operatorname{dom}\alpha}$ are well defined. Put $n_0=\max\left\{\textsf{N}_{\operatorname{dom}\alpha},\textsf{M}_{\operatorname{dom}\alpha}\right\}$. Then our assumption implies that there exists $(x_1,x_2,x_3)\in\operatorname{dom}\alpha\cap{\uparrow}(n_0,n_0,n_0)$ such that
\begin{equation*}
(x_1,x_2,x_3)\alpha=(x_1^{\alpha},x_2^{\alpha},x_3^{\alpha})\neq(x_1,x_2,x_3).
\end{equation*}
By statement $(i)$  we have that $(x_1^{\alpha},x_2^{\alpha},x_3^{\alpha})<(x_1,x_2,x_3)$. We consider the case when $x_1^{\alpha}<x_1$. In the cases when $x_2^{\alpha}<x_2$ or $x_3^{\alpha}<x_3$ the proofs are similar. We assume that $x_1\leq x_2$ and $x_1\leq x_3$. By statement $(i)$ the partial bijection $\alpha$ maps the set $S=\left\{(x,y,z)\in\mathbb{N}^3\colon x,y,z\leq x_1-1\right\}$ into itself. Also, by the definition of the semigroup $\mathscr{P\!O}\!_{\infty}(\mathbb{N}^3_{\leqslant})$ the partial bijection $\alpha$ maps the set
\begin{equation*}
\left\{(x_1,1,1), \ldots,(x_1,1,x_1),(x_1,2,1), \ldots,(x_1,2,x_1),\ldots, (x_1,x_1,1), \ldots,(x_1,x_1,x_1)\right\}
\end{equation*}
into $S$, too. Then our construction implies that
\begin{equation*}
\left|S\setminus\operatorname{dom}\alpha\right|=\left|\mathbb{N}^3\setminus\operatorname{dom}\alpha\right|=\textsf{N}_{\operatorname{dom}\alpha}-1
\end{equation*}
and
\begin{equation*}
  \left|\left\{(x_1{,}1{,}1){,} \ldots{,}(x_1{,}1{,}x_1){,}(x_1{,}2{,}1){,} \ldots{,}(x_1{,}2{,}x_1){,}\ldots{,}(x_1{,}x_1{,}1){,} \ldots{,}(x_1{,}x_1{,}x_1)\right\}\right|
    \geq\textsf{N}_{\operatorname{dom}\alpha},
\end{equation*}
a contradiction. In the case when $x_2\leq x_1$ and $x_2\leq x_3$ or $x_3\leq x_1$ and $x_3\leq x_2$ we get contradictions in similar ways. This completes the proof of existence of such a positive integer $n_{\alpha}$ for any $\alpha\in\mathscr{P\!O}\!_{\infty}(\mathbb{N}^3_{\leqslant})$. The existence of such minimal positive integer $n_{\alpha}$ follows from the fact that the set of all positive integers with the usual order $\leq$ is well-ordered.
\end{proof}

Theorem~\ref{theorem-2.10}$(iii)$ and Proposition~\ref{proposition-2.5} imply the following corollary.

\begin{corollary}\label{corollary-2.11}
For an arbitrary element $\alpha$ of the semigroup $\mathscr{P\!O}\!_{\infty}(\mathbb{N}^3_{\leqslant})$ there exist elements  $\sigma_1,\sigma_2$ of the group of units $H(\mathbb{I})$ of $\mathscr{P\!O}\!_{\infty}(\mathbb{N}^3_{\leqslant})$ and  a smallest positive integer $n_{\alpha}$ such that
$$
(x_1,x_2,x_3)\sigma_1\alpha=(x_1,x_2,x_3)\alpha\sigma_2=(x_1,x_2,x_3)
$$
for each $(x_1,x_2,x_3)\in\operatorname{dom}\alpha\cap{\uparrow}(n_{\alpha},n_{\alpha},n_{\alpha})$.
\end{corollary}

Corollary~\ref{corollary-2.11} implies

\begin{corollary}\label{corollary-2.12}
$\left|\mathbb{N}^3\setminus\operatorname{ran}\alpha\right|\leq\left|\mathbb{N}^3\setminus\operatorname{dom}\alpha\right|$ for an arbitrary $\alpha\in\mathscr{P\!O}\!_{\infty}(\mathbb{N}^3_{\leqslant})$.
\end{corollary}

%%%%%%%%%%%%%%%%%%%%%%%%%%%%%%%%%%%%%%%%%%%%%%%%%%%%%%%%%%%%%%%%%%

\section{Algebraic properties of the semigroup $\mathscr{P\!O}\!_{\infty}(\mathbb{N}^3_{\leqslant})$}\label{section-3}

\begin{proposition}\label{proposition-3.1}%{ }
Let $X$ be a non-empty set and let $\mathscr{P\!B}(X)$ be a semigroup of partial bijections of $X$ with the usual composition of partial self-maps.
Then an element $\alpha$ of $\mathscr{P\!B}(X)$ is an idempotent if and only if $\alpha$ is an identity partial self-map of $X$.
\end{proposition}

\begin{proof}[\textsl{Proof}]
The implication $(\Leftarrow)$ is trivial.

\smallskip

$(\Rightarrow)$ Let an element $\alpha$ be an idempotent of the semigroup $\mathscr{P\!B}(X)$. Then for every $x\in\operatorname{dom}\alpha$ we have that $(x)\alpha\alpha=(x)\alpha$ and hence we get that $\operatorname{dom}\alpha^2=\operatorname{dom}\alpha$ and $\operatorname{ran}\alpha^2=\operatorname{ran}\alpha$. Also since $\alpha$ is a partial bijective self-map of $X$ we conclude that the previous equalities imply that $\operatorname{dom}\alpha=\operatorname{ran}\alpha$. Fix an arbitrary $x\in\operatorname{dom}\alpha$ and suppose that $(x)\alpha=y$. Then $(x)\alpha=(x)\alpha\alpha=(y)\alpha=y$. Since $\alpha$ is a partial bijective self-map of the set $X$, we have that the equality $(y)\alpha=y$ implies that the full preimage of $y$ under the partial map $\alpha$ is equal to $y$. Similarly the equality $(x)\alpha=y$ implies that the full preimage of $y$ under the partial map $\alpha$ is equal to $x$. Thus we get that $x=y$ and our implication holds.
\end{proof}

Proposition~\ref{proposition-3.1} implies the following corollary.

\begin{corollary}\label{corollary-3.2}%{ }
An element $\alpha$ of $\mathscr{P\!O}\!_{\infty}(\mathbb{N}^n_{\leqslant})$ is an idempotent if and only if $\alpha$ is an identity partial self-map of $\mathbb{N}^n_{\leqslant}$ with the cofinite domain.
\end{corollary}

Corollary~\ref{corollary-3.2} implies the following proposition.

\begin{proposition}\label{proposition-3.3}%{ }
Let $n$ be a positive integer $\geqslant 2$.
The subset of idempotents $E(\mathscr{P\!O}\!_{\infty}(\mathbb{N}^n_{\leqslant}))$ of the semigroup $\mathscr{P\!O}\!_{\infty}(\mathbb{N}^n_{\leqslant})$ is a commutative submonoid of $\mathscr{P\!O}\!_{\infty}(\mathbb{N}^n_{\leqslant})$ and moreover $E(\mathscr{P\!O}\!_{\infty}(\mathbb{N}^n_{\leqslant}))$ is isomorphic to the free semilattice with unit $\left(\mathscr{P}^*(\mathbb{N}^n),\cup\right)$ over the set $\mathbb{N}^n$ under the map  $(\varepsilon)\mathfrak{h}=\mathbb{N}^n\setminus\operatorname{dom}\varepsilon$.
\end{proposition}

Later we shall need the following technical lemma.

\begin{lemma}\label{lemma-3.4}%{ }
Let $X$ be a non-empty set, $\mathscr{P\!B}(X)$ be the semigroup of partial bejections of $X$ with the usual composition of partial self-maps and $\alpha\in\mathscr{P\!B}(X)$. Then the following assertions hold:
\begin{itemize}
  \item[$(i)$] $\alpha=\gamma\alpha$ for some $\gamma\in\mathscr{P\!B}(X)$ if and only if the restriction $\gamma|_{\operatorname{dom}\alpha}\colon\operatorname{dom}\alpha\rightarrow X$ is an iden\-ti\-ty partial map;
  \item[$(ii)$] $\alpha=\alpha\gamma$ for some $\gamma\in\mathscr{P\!B}(X)$ if and only if the restriction $\gamma|_{\operatorname{ran}\alpha}\colon\operatorname{ran}\alpha\rightarrow X$ is an identity partial map.
\end{itemize}
\end{lemma}

\begin{proof}[\textsl{Proof}]
$(i)$ The implication $(\Leftarrow)$ is trivial.

$(\Rightarrow)$ Suppose that $\alpha=\gamma\alpha$ for some $\gamma\in\mathscr{P\!B}(X)$. Then $\operatorname{dom}\alpha\subseteq \operatorname{dom}\gamma$ and $\operatorname{dom}\alpha\subseteq \operatorname{ran}\gamma$. Since $\gamma\colon X\rightharpoonup X$ is a partial bijection, the above arguments imply that $(x)\gamma=x$ for each $x\in\operatorname{dom}\alpha$. Indeed,  if $(x)\gamma=y\neq x$ for some $y\in\operatorname{dom}\alpha$ then since $\alpha\colon X\rightharpoonup X$ is a partial bijection we have that either
\begin{equation*}
    (x)\alpha=(x)\gamma\alpha=(y)\alpha\neq(x)\alpha, \qquad \hbox{if} \quad y\in\operatorname{dom}\alpha,
\end{equation*}
or $(y)\alpha$ is undefined. This completes the proof of the implication.

The proof of $(ii)$ is similar to that of $(i)$.
\end{proof}

Lemma~\ref{lemma-3.4} implies the following corollary.

\begin{corollary}\label{corollary-3.5}%{ }
Let $n$ be a positive integer $\geqslant 2$ and $\alpha$ be an element of the semigroup $\mathscr{P\!O}\!_{\infty}(\mathbb{N}^n_{\leqslant})$. Then the following assertions hold:
\begin{itemize}
  \item[$(i)$] $\alpha=\gamma\alpha$ for some $\gamma\in\mathscr{P\!O}\!_{\infty}(\mathbb{N}^n_{\leqslant})$ if and only if the restriction $\gamma|_{\operatorname{dom}\alpha}\colon\operatorname{dom}\alpha\rightarrow \mathbb{N}^n$ is an iden\-ti\-ty partial map;
  \item[$(ii)$] $\alpha=\alpha\gamma$ for some $\gamma\in\mathscr{P\!O}\!_{\infty}(\mathbb{N}^n_{\leqslant})$ if and only if the restriction $\gamma|_{\operatorname{ran}\alpha}\colon\operatorname{ran}\alpha\rightarrow \mathbb{N}^n$ is an identity partial map.
\end{itemize}
\end{corollary}

The following theorem describes Green's relations $\mathscr{L}$, $\mathscr{R}$, $\mathscr{H}$ and $\mathscr{D}$ on the semigroup $\mathscr{P\!O}\!_{\infty}(\mathbb{N}^3_{\leqslant})$.

\begin{theorem}\label{theorem-3.6}%{ }
Let $\alpha$ and $\beta$ be elements of the semigroup $\mathscr{P\!O}\!_{\infty}(\mathbb{N}^3_{\leqslant})$. Then the following assertions hold:
\begin{itemize}
  \item[$(i)$] $\alpha\mathscr{L}\beta$ if and only if $\alpha=\mu\beta$ for some $\mu\in H(\mathbb{I})$;
  \item[$(ii)$] $\alpha\mathscr{R}\beta$ if and only if $\alpha=\beta\nu$ for some $\nu\in H(\mathbb{I})$;
  \item[$(iii)$] $\alpha\mathscr{H}\beta$ if and only if $\alpha=\mu\beta=\beta\nu$ for some $\mu,\nu\in H(\mathbb{I})$;
  \item[$(iv)$] $\alpha\mathscr{D}\beta$ if and only if $\alpha=\mu\beta\nu$ for some $\mu,\nu\in H(\mathbb{I})$.
\end{itemize}
\end{theorem}

\begin{proof}[\textsl{Proof}]
$(i)$ The implication $(\Leftarrow)$ is trivial.

$(\Rightarrow)$ Suppose that $\alpha\mathscr{L}\beta$ in the semigroup $\mathscr{P\!O}\!_{\infty}(\mathbb{N}^3_{\leqslant})$. Then there exist $\gamma,\delta\in\mathscr{P\!O}\!_{\infty}(\mathbb{N}^3_{\leqslant})$ such that $\alpha=\gamma\beta$ and $\beta=\delta\alpha$. The last equalities imply that
$\operatorname{ran}\alpha=\operatorname{ran}\beta$.

Next, we consider  the following cases:
\begin{itemize}
  \item[$(i_1)$] $(\mathscr{K}_i\cap\operatorname{dom}\alpha)\alpha\subseteq \mathscr{K}_i$ and $(\mathscr{K}_j\cap\operatorname{dom}\beta)\beta\subseteq \mathscr{K}_j$ for all $i,j=1,2,3$;
  \item[$(i_2)$] $(\mathscr{K}_i\cap\operatorname{dom}\alpha)\alpha\subseteq \mathscr{K}_i$ for all $i=1,2,3$ and $(\mathscr{K}_j\cap\operatorname{dom}\beta)\beta\nsubseteq \mathscr{K}_j$ for some $j=1,2,3$;
  \item[$(i_3)$] $(\mathscr{K}_i\cap\operatorname{dom}\alpha)\alpha\nsubseteq \mathscr{K}_i$ for some $i=1,2,3$ and $(\mathscr{K}_j\cap\operatorname{dom}\beta)\beta\subseteq \mathscr{K}_j$ for all $j=1,2,3$;
  \item[$(i_4)$] $(\mathscr{K}_i\cap\operatorname{dom}\alpha)\alpha\nsubseteq \mathscr{K}_i$ and $(\mathscr{K}_j\cap\operatorname{dom}\beta)\beta\nsubseteq \mathscr{K}_j$ for some $i,j=1,2,3$.
\end{itemize}

Suppose that case $(i_1)$ holds. Then Proposition~\ref{proposition-2.5} and the equalities $\alpha=\gamma\beta$ and $\beta=\delta\alpha$ imply that
\begin{equation}\label{eq-2.1}
    (\mathscr{K}_i\cap\operatorname{dom}\gamma)\gamma\subseteq \mathscr{K}_i \qquad \hbox{ and } \qquad
    (\mathscr{K}_j\cap\operatorname{dom}\delta)\delta\subseteq \mathscr{K}_j, \quad \hbox{for all} \quad i,j=1,2,3,
\end{equation}
and moreover we have that $\alpha=\gamma\delta\alpha$ and $\beta=\delta\gamma\beta$. Hence by Lemma~\ref{lemma-3.4} we have that the restrictions $(\gamma\delta)|_{\operatorname{dom}\alpha}\colon\operatorname{dom}\alpha\rightharpoonup \mathbb{N}^3$ and $(\delta\gamma)|_{\operatorname{dom}\beta}\colon\operatorname{dom}\beta\rightharpoonup \mathbb{N}^3$ are identity partial maps. Then by condition $(\ref{eq-2.1})$ we obtain that the restrictions $\gamma|_{\operatorname{dom}\alpha}\colon\operatorname{dom}\alpha\rightharpoonup \mathbb{N}^3$ and $\delta|_{\operatorname{dom}\beta}\colon\operatorname{dom}\beta\rightharpoonup \mathbb{N}^3$ are identity partial maps, as well. Indeed, otherwise there exists $\overline{x}\in\operatorname{dom}\alpha$ such that either $(\overline{x})\gamma\nleqslant\overline{x}$ or $(\overline{x})\delta\nleqslant\overline{x}$, which contradicts Theorem~\ref{theorem-2.10}$(ii)$. Thus, the above arguments imply that in case $(i_1)$ we have the  equality $\alpha=\beta$.

\smallskip

Suppose that case $(i_2)$ holds. By Corollary~\ref{corollary-2.6} there exists an element $\mu$ of the group of units $H(\mathbb{I})$ of the semigroup $\mathscr{P\!O}\!_{\infty}(\mathbb{N}^3_{\leqslant})$ such that $(\mathscr{K}_j\cap\operatorname{dom}\beta)\mu\beta\subseteq \mathscr{K}_j$ for all $j=1,2,3$, and, since $\alpha\mathscr{L}\beta$,
we have that
\begin{equation*}
\alpha=\gamma\beta=\gamma\mathbb{I}\beta=\gamma(\mu^{-1}\mu)\beta=(\gamma\mu^{-1})(\mu\beta)
\end{equation*}
and $\mu\beta=(\mu\delta)\alpha$. Hence we get that $\alpha\mathscr{L}(\mu\beta)$, $(\mathscr{K}_i\cap\operatorname{dom}\alpha)\alpha\subseteq \mathscr{K}_i$ and $(\mathscr{K}_j\cap\operatorname{dom}\beta)\mu\beta\subseteq \mathscr{K}_j$ for all $i,j=1,2,3$. Then we apply case $(i_1)$ for the elements $\alpha$ and $\mu\beta$ and obtain the equality $\alpha=\mu\beta$, where $\mu$ is the above determined element of the group of units $H(\mathbb{I})$.

\smallskip

In case $(i_3)$ the proof of the equality $\alpha=\mu\beta$ is similar to case $(i_2)$.

\smallskip

Suppose that case $(i_4)$ holds. By Corollary~\ref{corollary-2.6} there exist elements $\mu_\alpha$ and $\mu_\beta$ of the group of units $H(\mathbb{I})$ of the semigroup $\mathscr{P\!O}\!_{\infty}(\mathbb{N}^3_{\leqslant})$ such that $(\mathscr{K}_j\cap\operatorname{dom}\alpha)\mu_\alpha\alpha\subseteq \mathscr{K}_j$ and $(\mathscr{K}_j\cap\operatorname{dom}\beta)\mu_\beta\beta\subseteq \mathscr{K}_j$ for all $i,j=1,2,3$, and, since $\alpha\mathscr{L}\beta$,
we have that
\begin{equation*}
\alpha=\gamma\beta=\gamma\mathbb{I}\beta=\gamma(\mu^{-1}_\beta\mu_\beta)\beta=(\gamma\mu^{-1}_\beta)(\mu_\beta\beta)
\end{equation*}
and
\begin{equation*}
\beta=\delta\alpha=\delta\mathbb{I}\alpha=\delta(\mu^{-1}_\alpha\mu_\alpha)\alpha=(\delta\mu^{-1}_\alpha)(\mu_\alpha\alpha).
\end{equation*}
Hence we get that
\begin{equation*}
\mu_\alpha\alpha=(\mu_\alpha\gamma\mu^{-1}_\beta)(\mu_\beta\beta)\qquad \hbox{and} \qquad \mu_\beta\beta=(\mu_\beta\delta\mu^{-1}_\alpha)(\mu_\alpha\alpha).
\end{equation*}
The last two equalities imply that $(\mu_\beta\beta)\mathscr{L}(\mu_\alpha\alpha)$ and by above part of the proof we have that $(\mathscr{K}_j\cap\operatorname{dom}\alpha)\mu_\alpha\alpha\subseteq \mathscr{K}_j$ and $(\mathscr{K}_j\cap\operatorname{dom}\beta)\mu_\beta\beta\subseteq \mathscr{K}_j$ for all $i,j=1,2,3$. Then we apply case $(i_1)$ for the elements $\mu_\alpha\alpha$ and $\mu_\beta\beta$ and obtain the equality $\mu_\alpha\alpha=\mu_\beta\beta$. Hence $\alpha=\mu_\alpha^{-1}\mu_\alpha\alpha=\mu_\alpha^{-1}\mu_\beta\beta$. Since $\mu_\alpha,\mu_\alpha\in H(\mathbb{I})$, $\mu=\mu_\alpha^{-1}\mu_\beta\in H(\mathbb{I})$ as well.

\smallskip

The proof of assertion $(ii)$ is dual to that of $(i)$.

\smallskip

Assertion $(iii)$ follows from $(i)$ and $(ii)$.

\smallskip

$(iv)$ Suppose that $\alpha\mathscr{D}\beta$ in $\mathscr{P\!O}\!_{\infty}(\mathbb{N}^3_{\leqslant})$. Then there exists $\gamma\in\mathscr{P\!O}\!_{\infty}(\mathbb{N}^3_{\leqslant})$ such that $\alpha\mathscr{L}\gamma$ and $\gamma\mathscr{R}\beta$.  By statements $(i)$ and  $(ii)$  there exist $\mu,\nu\in H(\mathbb{I})$ such that $\alpha=\mu\gamma$ and $\gamma=\beta\nu$ and hence $\alpha=\mu\beta\nu$. Converse, suppose that $\alpha=\mu\beta\nu$ for some $\mu,\nu\in H(\mathbb{I})$. Then by $(i)$,  $(ii)$, we have that $\alpha\mathscr{L}(\beta\nu)$ and $(\beta\nu)\mathscr{R}\beta$, and hence $\alpha\mathscr{D}\beta$ in $\mathscr{P\!O}\!_{\infty}(\mathbb{N}^3_{\leqslant})$.
\end{proof}

Theorem~\ref{theorem-3.6} implies Corollary~\ref{corollary-3.7} which gives the inner characterization of Green's relations $\mathscr{L}$, $\mathscr{R}$, and $\mathscr{H}$ %and $\mathscr{D}$
on the semigroup $\mathscr{P\!O}\!_{\infty}(\mathbb{N}^3_{\leqslant})$ as partial permutations of the poset $\mathbb{N}^3_{\leqslant}$.

\begin{corollary}\label{corollary-3.7}%{ }
\begin{itemize}
  \item[$(i)$] Every $\mathscr{L}$-class of $\mathscr{P\!O}\!_{\infty}(\mathbb{N}^3_{\leqslant})$ contains exactly $6$ distinct elements.
  \item[$(ii)$] Every $\mathscr{R}$-class of $\mathscr{P\!O}\!_{\infty}(\mathbb{N}^3_{\leqslant})$ contains exactly  $6$ distinct elements.
  \item[$(iii)$] Every $\mathscr{H}$-class of $\mathscr{P\!O}\!_{\infty}(\mathbb{N}^3_{\leqslant})$ contains at most $6$ distinct elements.
  \end{itemize}
\end{corollary}

\begin{proof}[\textsl{Proof}]
Statements $(i)$, $(ii)$ and $(iii)$  are trivial and they follow from the corresponding statements of Theorem~\ref{theorem-3.6}.
\end{proof}

\begin{lemma}\label{lemma-3.8}%{ }
Let $\alpha,\beta$ and $\gamma$ be elements of the semigroup $\mathscr{P\!O}\!_{\infty}(\mathbb{N}^3_{\leqslant})$ such that $\alpha=\beta\alpha\gamma$. Then the following statements hold:
\begin{itemize}
  \item[$(i)$] if $(\mathscr{K}_i\cap\operatorname{dom}\beta)\beta\subseteq \mathscr{K}_i$  for any $i=1,2,3$, then the restrictions $\beta|_{\operatorname{dom}\alpha}\colon \operatorname{dom}\alpha\rightharpoonup \mathbb{N}^3$ and $\gamma|_{\operatorname{ran}\alpha}\colon \operatorname{ran}\alpha\rightharpoonup \mathbb{N}^3$ are identity partial maps;

  \item[$(ii)$] if $(\mathscr{K}_i\cap\operatorname{dom}\gamma)\gamma\subseteq \mathscr{K}_i$  for any $i=1,2,3$, then the restrictions $\beta|_{\operatorname{dom}\alpha}\colon \operatorname{dom}\alpha\rightharpoonup \mathbb{N}^3$ and $\gamma|_{\operatorname{ran}\alpha}\colon \operatorname{ran}\alpha\rightharpoonup \mathbb{N}^3$ are identity partial maps;
  \item[$(iii)$] there exist elements $\sigma_\beta$ and $\sigma_\gamma$ of the group of units $H(\mathbb{I})$ of $\mathscr{P\!O}\!_{\infty}(\mathbb{N}^3_{\leqslant})$ such that $\alpha=\sigma_\beta\alpha\sigma_\gamma$.
\end{itemize}
\end{lemma}

\begin{proof}[\textsl{Proof}]
$(i)$ Assume that the inclusion $(\mathscr{K}_i\cap\operatorname{dom}\beta)\beta\subseteq \mathscr{K}_i$ holds for any $i=1,2,3$. Then one of the following cases holds:
\begin{itemize}
  \item[$(1)$] $(\mathscr{K}_i\cap\operatorname{dom}\alpha)\alpha\subseteq \mathscr{K}_i$  for any $i=1,2,3$;
  \item[$(2)$] there exists $i\in\{1,2,3\}$ such that $(\mathscr{K}_i\cap\operatorname{dom}\alpha)\alpha\nsubseteq \mathscr{K}_i$.
\end{itemize}

If case $(1)$ holds then the equality $\alpha=\beta\alpha\gamma$ and Proposition~\ref{proposition-2.5} imply that $(\mathscr{K}_i\cap\operatorname{dom}\gamma)\gamma\subseteq \mathscr{K}_i$  for any $i=1,2,3$.  Suppose that $(\overline{x})\beta<\overline{x}$ for some $\overline{x}\in\operatorname{dom}\alpha$. Then by  Theorem~\ref{theorem-2.10}$(i)$ we have that
\begin{equation*}
    (\overline{x})\alpha=(\overline{x})\beta\alpha\gamma<(\overline{x})\alpha\gamma\leqslant(\overline{x})\alpha,
\end{equation*}
which contradicts the equality $\alpha=\beta\alpha\gamma$. The obtained contradiction implies that the restriction $\beta|_{\operatorname{dom}\alpha}\colon \operatorname{dom}\alpha\rightharpoonup \mathbb{N}^3$ is an identity partial map. This and the equality $\alpha=\beta\alpha\gamma$ imply that the restriction $\gamma|_{\operatorname{ran}\alpha}\colon \operatorname{ran}\alpha\rightharpoonup \mathbb{N}^3$ is an identity partial map too.

Suppose that case $(2)$ holds. Then by Corollary~\ref{corollary-2.6} there exists an element $\sigma$ of the group of units $H(\mathbb{I})$ of the semigroup $\mathscr{P\!O}\!_{\infty}(\mathbb{N}^3_{\leqslant})$ such that $(\mathscr{K}_i\cap\operatorname{dom}\alpha)\alpha\sigma\subseteq \mathscr{K}_i$  for any $i=1,2,3$. Now, the equality $\alpha=\beta\alpha\gamma$ implies that
\begin{equation*}
    \alpha\sigma=\beta\alpha\gamma\sigma=\beta\alpha\mathbb{I}\gamma\sigma=\beta\alpha(\sigma\sigma^{-1})\gamma\sigma=\beta(\alpha\sigma)(\sigma^{-1}\gamma\sigma).
\end{equation*}
By case $(1)$ we have that the restrictions $\beta|_{\operatorname{dom}\alpha}\colon \operatorname{dom}\alpha\rightharpoonup \mathbb{N}^3$ is an identity partial map, which implies that $\beta\alpha=\alpha$. Then we have that $\alpha=\beta\alpha\gamma=\alpha\gamma$ and hence by Corollary~\ref{corollary-3.5} the restriction $\gamma|_{\operatorname{ran}\alpha}\colon \operatorname{ran}\alpha\rightharpoonup \mathbb{N}^3$ is an identity partial map, which completes the proof of statement $(i)$.

\smallskip

$(ii)$ The proof of this statement is dual to $(i)$. Indeed, assume that the inclusion $(\mathscr{K}_i\cap\operatorname{dom}\gamma)\gamma\subseteq \mathscr{K}_i$ holds for any $i=1,2,3$. Then one of the following cases holds:
\begin{itemize}
  \item[$(1)$] $(\mathscr{K}_i\cap\operatorname{dom}\alpha)\alpha\subseteq \mathscr{K}_i$  for any $i=1,2,3$;
  \item[$(2)$] there exists $i\in\{1,2,3\}$ such that $(\mathscr{K}_i\cap\operatorname{dom}\alpha)\alpha\nsubseteq \mathscr{K}_i$.
\end{itemize}

If case $(1)$ holds then the equality $\alpha=\beta\alpha\gamma$ and Proposition~\ref{proposition-2.5} imply that $(\mathscr{K}_i\cap\operatorname{dom}\beta)\beta\subseteq \mathscr{K}_i$  for any $i=1,2,3$. Similarly as in the proof of statement $(i)$ Theorem~\ref{theorem-2.10}$(i)$ implies that the restrictions $\beta|_{\operatorname{dom}\alpha}\colon \operatorname{dom}\alpha\rightharpoonup \mathbb{N}^3$ and $\gamma|_{\operatorname{ran}\alpha}\colon \operatorname{ran}\alpha\rightharpoonup \mathbb{N}^3$ are identity partial maps.

\smallskip

Suppose that case $(2)$ holds. Then by Corollary~\ref{corollary-2.6} there exists an element $\sigma$ of the group of units $H(\mathbb{I})$ of the semigroup $\mathscr{P\!O}\!_{\infty}(\mathbb{N}^3_{\leqslant})$ such that $(\mathscr{K}_i\cap\operatorname{dom}\alpha)\sigma\alpha\subseteq \mathscr{K}_i$  for any $i=1,2,3$. Now, the equality $\alpha=\beta\alpha\gamma$ implies that
\begin{equation*}
    \sigma\alpha=\sigma\beta\alpha\gamma=\sigma\beta\mathbb{I}\alpha\gamma=\sigma\beta(\sigma^{-1}\sigma)\alpha\gamma=(\sigma\beta\sigma^{-1})(\sigma\alpha)\gamma.
\end{equation*}
By case $(1)$ we have that the restriction $\gamma|_{\operatorname{ran}\alpha}\colon \operatorname{ran}\alpha\rightharpoonup \mathbb{N}^3$ is an identity partial map, which implies that $\alpha=\alpha\gamma$. Then we have that $\alpha=\beta\alpha\gamma=\beta\alpha$ and hence by Corollary~\ref{corollary-3.5} the restriction $\beta|_{\operatorname{dom}\alpha}\colon \operatorname{dom}\alpha\rightharpoonup \mathbb{N}^3$ is an identity partial map as well, which completes the proof of statement $(ii)$.

\smallskip

$(iii)$ Assume that $\alpha=\beta\alpha\gamma$. By the Lagrange Theorem (see: \cite[Section~1.5]{Hall-1963}) for every element $\sigma$ of the group of permutations $\mathscr{S}_n$ the order of $\sigma$ divides the order of $\mathscr{S}_n$. This, Proposition~\ref{proposition-2.5} and the equality $\alpha=\beta\alpha\gamma$ imply that
\begin{equation}\label{eq-3.1}
  (\mathscr{K}_i\cap\operatorname{dom}\beta^6)\beta^6\subseteq \mathscr{K}_{i} \qquad \hbox{and} \qquad (\mathscr{K}_i\cap\operatorname{dom}\gamma^6)\gamma^6\subseteq \mathscr{K}_{i}, \qquad \hbox{for any} \quad i=1,2,3.
\end{equation}
Also, the equality $\alpha=\beta\alpha\gamma$ implies that
\begin{equation*}\label{eq-3.2}
  \alpha=\beta\alpha\gamma=\beta(\beta\alpha\gamma)\gamma=\beta^2\alpha\gamma^2=\ldots=\beta^6\alpha\gamma^6.
\end{equation*}
Then statements $(i)$, $(ii)$ and conditions \eqref{eq-3.1} imply that the restrictions $\beta^6|_{\operatorname{dom}\alpha}\colon \operatorname{dom}\alpha\rightharpoonup \mathbb{N}^3$ and $\gamma^6|_{\operatorname{ran}\alpha}\colon \operatorname{ran}\alpha\rightharpoonup \mathbb{N}^3$ are identity partial maps. By Corollary~\ref{corollary-2.6} there exist  unique elements $\sigma_{\beta},\sigma_{\gamma}\in H(\mathbb{I})$ such that $(\mathscr{K}_i\cap\operatorname{dom}\beta)\beta\sigma^{-1}_{\beta}\subseteq \mathscr{K}_i$, $(\mathscr{K}_i\cap \operatorname{dom}\beta)\sigma_{\beta}\beta\subseteq \mathscr{K}_i$, $(\mathscr{K}_i\cap\operatorname{dom}\alpha)\gamma\sigma^{-1}_{\gamma}\subseteq \mathscr{K}_i$ and $(\mathscr{K}_i\cap \operatorname{dom}\gamma)\sigma_{\gamma}\gamma\subseteq \mathscr{K}_i$ for all $i=1,2,3$. Then we have that
\begin{equation}\label{eq-3.3}
\begin{split}
  \beta^6 & =(\beta\mathbb{I}\beta)(\beta\mathbb{I}\beta)(\beta\mathbb{I}\beta) \\
    & =(\beta\sigma^{-1}_{\beta}\sigma_{\beta}\beta)(\beta\sigma^{-1}_{\beta}\sigma_{\beta}\beta)(\beta\sigma^{-1}_{\beta}\sigma_{\beta}\beta)\\
    & =(\beta\sigma^{-1}_{\beta})(\sigma_{\beta}\beta)(\beta\sigma^{-1}_{\beta})(\sigma_{\beta}\beta)(\beta\sigma^{-1}_{\beta})(\sigma_{\beta}\beta)
\end{split}
\end{equation}
and
\begin{equation}\label{eq-3.4}
\begin{split}
  \gamma^6 & =(\gamma\mathbb{I}\gamma)(\gamma\mathbb{I}\gamma)(\gamma\mathbb{I}\gamma) \\
    & =(\gamma\sigma^{-1}_{\gamma}\sigma_{\gamma}\gamma)(\gamma\sigma^{-1}_{\gamma}\sigma_{\gamma}\gamma)(\gamma\sigma^{-1}_{\gamma}\sigma_{\gamma}\gamma) \\
    & =(\gamma\sigma^{-1}_{\gamma})(\sigma_{\gamma}\gamma)(\gamma\sigma^{-1}_{\gamma})(\sigma_{\gamma}\gamma)
    (\gamma\sigma^{-1}_{\gamma})(\sigma_{\gamma}\gamma).
\end{split}
\end{equation}

We claim that $(\overline{x})(\beta\sigma^{-1}_{\beta})=\overline{x}$  for any $\overline{x}\in\operatorname{dom}\alpha$. Assume that $(\overline{x})(\beta\sigma^{-1}_{\beta})\neq\overline{x}$ for some $\overline{x}\in\operatorname{dom}\alpha$. Then the choice of the element $\sigma_{\beta}\in H(\mathbb{I})$, Theorem~\ref{theorem-2.10}$(i)$  and \eqref{eq-3.3} imply that
\begin{equation*}
\begin{split}
  (\overline{x})\beta^6 & = (\overline{x})(\beta\sigma^{-1}_{\beta})(\sigma_{\beta}\beta)(\beta\sigma^{-1}_{\beta})(\sigma_{\beta}\beta)(\beta\sigma^{-1}_{\beta})
  (\sigma_{\beta}\beta)\\
    & <(\overline{x})(\sigma_{\beta}\beta)(\beta\sigma^{-1}_{\beta})(\sigma_{\beta}\beta)(\beta\sigma^{-1}_{\beta})(\sigma_{\beta}\beta) \\
    & \leqslant (\overline{x})(\beta\sigma^{-1}_{\beta})(\sigma_{\beta}\beta)(\beta\sigma^{-1}_{\beta})(\sigma_{\beta}\beta)\\
    & < (\overline{x})(\sigma_{\beta}\beta)(\beta\sigma^{-1}_{\beta})(\sigma_{\beta}\beta)\\
    & \leqslant (\overline{x})(\beta\sigma^{-1}_{\beta})(\sigma_{\beta}\beta)\\
    & < (\overline{x})(\sigma_{\beta}\beta)\\
    & \leqslant\overline{x},
\end{split}
\end{equation*}
which  contradicts the fact that the restriction $\beta^6|_{\operatorname{dom}\alpha}\colon \operatorname{dom}\alpha\rightharpoonup \mathbb{N}^3$ is an identity partial map. Hence we have that $(\overline{x})(\beta\sigma^{-1}_{\beta})=\overline{x}$ for any $\overline{x}\in\operatorname{dom}\alpha$, which implies that the equality $(\overline{x})\beta=(\overline{x})\sigma_{\beta}$ holds for any $\overline{x}\in\operatorname{dom}\alpha$.

Using \eqref{eq-3.4} as in the above we prove the equality $(\overline{x})\gamma=(\overline{x})\sigma_{\gamma}$ holds for any $\overline{x}\in\operatorname{ran}\alpha$.

The obtained equalities and the definition of the composition of partial maps imply statement $(iii)$.
\end{proof}

\begin{lemma}\label{lemma-3.9}%{ }
Let $\alpha$ and $\beta$ be elements of the semigroup $\mathscr{P\!O}\!_{\infty}(\mathbb{N}^3_{\leqslant})$ and $A$ be a cofinite subset of $\mathbb{N}^3$. If the restriction $(\alpha\beta)|_A\colon A\rightharpoonup \mathbb{N}^3$ is an identity partial map then there exists an element $\sigma$ of the group of units $H(\mathbb{I})$ of $\mathscr{P\!O}\!_{\infty}(\mathbb{N}^3_{\leqslant})$ such that $(\overline{x})\alpha=(\overline{x})\sigma$ and $(\overline{y})\beta=(\overline{y})\sigma^{-1}$ for all $\overline{x}\in A$ and $\overline{y}\in (A)\alpha$.
\end{lemma}

\begin{proof}[\textsl{Proof}]
We observe that  one of the following cases holds:
\begin{itemize}
  \item[$(1)$] $(\mathscr{K}_i\cap A)\alpha\subseteq \mathscr{K}_i$  for any $i=1,2,3$;
  \item[$(2)$] there exists $i\in\{1,2,3\}$ such that $(\mathscr{K}_i\cap A)\alpha\nsubseteq \mathscr{K}_i$.
\end{itemize}

If case $(1)$ holds then the assumption of the lemma and Proposition~\ref{proposition-2.5} imply that $(\mathscr{K}_i\cap(A)\alpha)\beta\subseteq \mathscr{K}_i$  for any $i=1,2,3$. Suppose that $(\overline{x})\alpha<\overline{x}$ for some $\overline{x}\in A$. Then by Theorem~\ref{theorem-2.10}$(i)$ we have that
\begin{equation*}
    (\overline{x})\alpha\beta<(\overline{x})\beta\leqslant \overline{x},
\end{equation*}
which contradicts the assumption of the lemma. Similarly we show that the case $(\overline{y})\beta<\overline{y}$ for some $\overline{y}\in (A)\alpha$ does not hold. The obtained contradiction implies that $(\overline{x})\alpha=\overline{x}$ and $(\overline{x})\beta=\overline{x}$ for all $\overline{x}\in A$.

Suppose that case $(2)$ holds. Then by Corollary~\ref{corollary-2.6} there exists an element $\sigma$ of the group of units $H(\mathbb{I})$ of the semigroup $\mathscr{P\!O}\!_{\infty}(\mathbb{N}^3_{\leqslant})$ such that $(\mathscr{K}_i\cap\operatorname{dom}\alpha)\alpha\sigma\subseteq \mathscr{K}_i$  for any $i=1,2,3$. Now,  the assumption of the lemma implies that
\begin{equation*}
    (\overline{x})\alpha\beta=(\overline{x})\alpha\mathbb{I}\beta=(\overline{x})\alpha\sigma\sigma^{-1}\beta=\overline{x},
\end{equation*}
and hence by the above part of the proof we get that $(\overline{x})\alpha\sigma=\overline{x}$ and $(\overline{y})\sigma^{-1}\beta=\overline{x}$ for all $\overline{y}\in (A)\alpha$. The obtained equalities and the definition of the composition of partial maps imply the statement of the lemma.
\end{proof}

\begin{lemma}\label{lemma-3.10}%{ }
Let $\alpha$, $\beta$, $\gamma$ and $\delta$ be elements of the semigroup $\mathscr{P\!O}\!_{\infty}(\mathbb{N}^3_{\leqslant})$ such that $\alpha=\gamma\beta\delta$. Then there exist $\gamma^*,\delta^*\in\mathscr{P\!O}\!_{\infty}(\mathbb{N}^3_{\leqslant})$ such that $\alpha=\gamma^*\beta\delta^*$,  $\operatorname{dom}\gamma^*=\operatorname{dom}\alpha$, $\operatorname{ran}\gamma^*=\operatorname{dom}\beta$, $\operatorname{dom}\delta^*=\operatorname{ran}\beta$ and $\operatorname{ran}\delta^*=\operatorname{ran}\alpha$.
\end{lemma}

\begin{proof}[\textsl{Proof}]
For a cofinite subset $A$ of $\mathbb{N}^3$ by $\iota_A$ we denote the identity map of $A$. It is obvious that $\iota_A\in\mathscr{P\!O}\!_{\infty}(\mathbb{N}^3_{\leqslant})$ for any cofinite subset $A$ of $\mathbb{N}^3$. This implies that $\alpha=\iota_{\operatorname{dom}\alpha}\alpha\iota_{\operatorname{ran}\alpha}$ and $\beta=\iota_{\operatorname{dom}\beta}\beta\iota_{\operatorname{ran}\beta}$, and hence we have that
\begin{equation*}
  \alpha=\iota_{\operatorname{dom}\alpha}\alpha\iota_{\operatorname{ran}\alpha}= \iota_{\operatorname{dom}\alpha}\gamma\beta\delta\iota_{\operatorname{ran}\alpha}= \iota_{\operatorname{dom}\alpha}\gamma\iota_{\operatorname{dom}\beta}\beta\iota_{\operatorname{ran}\beta}\delta\iota_{\operatorname{ran}\alpha}.
\end{equation*}
We put $\gamma^*=\iota_{\operatorname{dom}\alpha}\gamma\iota_{\operatorname{dom}\beta}$ and $\delta^*=\iota_{\operatorname{ran}\beta}\delta\iota_{\operatorname{ran}\alpha}$. The above two equalities and the definition of the semigroup operation of $\mathscr{P\!O}\!_{\infty}(\mathbb{N}^3_{\leqslant})$ imply that $\operatorname{dom}\gamma^*\subseteq\operatorname{dom}\alpha$, $\operatorname{ran}\gamma^*\subseteq\operatorname{dom}\beta$, $\operatorname{dom}\delta^*\subseteq\operatorname{ran}\beta$ and $\operatorname{ran}\delta^*\subseteq\operatorname{ran}\alpha$. Similar arguments and the equality $\alpha= \gamma^*\beta\delta^*$ imply the converse inclusions which implies the statement of the lemma.
\end{proof}

\begin{theorem}\label{theorem-3.12}%{ }
$\mathscr{D}=\mathscr{J}$ in $\mathscr{P\!O}\!_{\infty}(\mathbb{N}^3_{\leqslant})$.
\end{theorem}

\begin{proof}[\textsl{Proof}]
The inclusion $\mathscr{D}\subseteq\mathscr{J}$ is trivial.

Fix any $\alpha,\beta\in\mathscr{P\!O}\!_{\infty}(\mathbb{N}^3_{\leqslant})$ such that $\alpha\mathscr{J}\beta$. Then there exist $\gamma_\alpha,\delta_\alpha, \gamma_\beta,\delta_\beta\in\mathscr{P\!O}\!_{\infty}(\mathbb{N}^3_{\leqslant})$ such that $\alpha=\gamma_\alpha\beta\delta_\alpha$ and $\beta=\gamma_\beta\alpha\delta_\beta$ (see \cite{GreenJ-1951} or \cite[Section~II.1]{Grillet-1995}). By Lemma~\ref{lemma-3.10} without loss of generality we may assume that
\begin{equation*}
  \operatorname{dom}\gamma_{\alpha}=\operatorname{dom}\alpha, \qquad \operatorname{ran}\gamma_{\alpha}=\operatorname{dom}\beta, \qquad \operatorname{dom}\delta_{\alpha}=\operatorname{ran}\beta, \qquad \operatorname{ran}\delta_{\alpha}=\operatorname{ran}\alpha
\end{equation*}
and
\begin{equation*}
  \operatorname{dom}\gamma_{\beta}=\operatorname{dom}\beta, \qquad \operatorname{ran}\gamma_{\beta}=\operatorname{dom}\alpha, \qquad \operatorname{dom}\delta_{\beta}=\operatorname{ran}\alpha, \qquad \operatorname{ran}\delta_{\beta}=\operatorname{ran}\beta.
\end{equation*}
Hence we have that $\alpha=\gamma_\alpha\gamma_\beta\alpha\delta_\beta\delta_\alpha$ and $\beta=\gamma_\beta\gamma_\alpha\beta\delta_\alpha\delta_\beta$. Then only one of the following cases holds:
\begin{itemize}
  \item[$(1)$] $(\mathscr{K}_i\cap \operatorname{dom}(\gamma_\alpha\gamma_\beta))\gamma_\alpha\gamma_\beta\subseteq \mathscr{K}_i$  for any $i=1,2,3$;
  \item[$(2)$] there exists $i\in\{1,2,3\}$ such that $(\mathscr{K}_i\cap \operatorname{dom}(\gamma_\alpha\gamma_\beta))\gamma_\alpha\gamma_\beta\nsubseteq \mathscr{K}_i$.
\end{itemize}

If case $(1)$ holds then Lemma~\ref{lemma-3.8}$(i)$ implies that $(\gamma_\alpha\gamma_\beta)\colon \operatorname{dom}\alpha\rightharpoonup \mathbb{N}^3$ and $(\delta_\beta\delta_\alpha)\colon \operatorname{ran}\alpha\rightharpoonup \mathbb{N}^3$ are identity partial maps. Now by Lemma~\ref{lemma-3.9} there exist  elements $\sigma_\alpha$ and $\sigma_\beta$ of the group of units $H(\mathbb{I})$ of the semigroup $\mathscr{P\!O}\!_{\infty}(\mathbb{N}^3_{\leqslant})$ such that $(\overline{x})\gamma_\alpha=(\overline{x})\sigma_\alpha$, $(\overline{y})\gamma_\beta=(\overline{y})\sigma^{-1}_\alpha$, $(\overline{u})\delta_\beta=(\overline{u})\sigma_\beta$ and $(\overline{v})\delta_\alpha=(\overline{v})\sigma^{-1}_\beta$,  for all $\overline{x}\in \operatorname{dom}\alpha$, $\overline{y}\in (\operatorname{dom}\alpha)\gamma_\alpha=\operatorname{ran}\gamma_{\alpha}=\operatorname{dom}\beta$, $\overline{u}\in \operatorname{ran}\alpha$ and $\overline{v}\in (\operatorname{ran}\alpha)\delta_\beta=\operatorname{ran}\delta_{\beta}=\operatorname{ran}\beta$. Then the above arguments imply that $\alpha=\sigma_\alpha\beta\sigma^{-1}_\beta$ and hence by Theorem~\ref{theorem-3.6}$(iv)$ we get that $\alpha\mathscr{D}\beta$ in $\mathscr{P\!O}\!_{\infty}(\mathbb{N}^3_{\leqslant})$.

If case $(2)$ holds then we have that
\begin{equation*}
  \alpha=\gamma_\alpha\gamma_\beta\alpha\delta_\beta\delta_\alpha=(\gamma_\alpha\gamma_\beta)^2\alpha(\delta_\beta\delta_\alpha)^2=\ldots= (\gamma_\alpha\gamma_\beta)^6\alpha(\delta_\beta\delta_\alpha)^6
\end{equation*}
and
\begin{equation*}
  \beta=\gamma_\beta\gamma_\alpha\beta\delta_\alpha\delta_\beta=(\gamma_\beta\gamma_\alpha)^2\beta(\delta_\alpha\delta_\beta)^2=\ldots=
  (\gamma_\beta\gamma_\alpha)^6\beta(\delta_\alpha\delta_\beta)^6.
\end{equation*}
We put
\begin{equation*}
  \gamma_\beta^\circ=\gamma_\beta(\gamma_\alpha\gamma_\beta)^5 \qquad \hbox{and} \qquad \delta_\beta^\circ=\delta_\beta(\delta_\alpha\delta_\beta)^5.
\end{equation*}
Lemma~\ref{lemma-3.8}$(i)$ implies that $(\gamma_\alpha\gamma_\beta^\circ)\colon \operatorname{dom}\alpha\rightharpoonup \mathbb{N}^3$ and $(\delta_\beta^\circ\delta_\alpha)\colon \operatorname{ran}\alpha\rightharpoonup \mathbb{N}^3$ are identity partial maps. Now by Lemma~\ref{lemma-3.9} there exist  elements $\sigma_\alpha$ and $\sigma_\beta$ of the group of units $H(\mathbb{I})$ of the semigroup $\mathscr{P\!O}\!_{\infty}(\mathbb{N}^3_{\leqslant})$ such that $(\overline{x})\gamma_\alpha=(\overline{x})\sigma_\alpha$, $(\overline{y})\gamma_\beta^\circ=(\overline{y})\sigma^{-1}_\alpha$, $(\overline{u})\delta_\beta^\circ=(\overline{u})\sigma_\beta$ and $(\overline{v})\delta_\alpha=(\overline{v})\sigma^{-1}_\beta$,  for all $\overline{x}\in \operatorname{dom}\alpha$, $\overline{y}\in (\operatorname{dom}\alpha)\gamma_\alpha=\operatorname{ran}\gamma_{\alpha}=\operatorname{dom}\beta$, $\overline{u}\in \operatorname{ran}\alpha$ and $\overline{v}\in (\operatorname{ran}\alpha)\delta_\beta^\circ=\operatorname{ran}\delta_{\beta}^\circ=\operatorname{ran}\beta$. Then the above arguments imply that $\alpha=\sigma_\alpha\beta\sigma^{-1}_\beta$ and hence by Theorem~\ref{theorem-3.6}$(iv)$ we get that $\alpha\mathscr{D}\beta$ in $\mathscr{P\!O}\!_{\infty}(\mathbb{N}^3_{\leqslant})$.
\end{proof}

%%%%%%%%%%%%%%%%%%%%%%%%%%%%%%%%%%%%%%%%%%%%%%%%%%%%%%%%%%%

%%%%%%%%%%%%%%%%%%%%%%%%%%%%%%%%%%%%%%%%%%%%%%%%%%%%%%%%%%%%%%%
%%%%%%%%%%%%%%%%%%%%%%%%%%%%%%%%%%%%%%%%%%%%%%%%%%%%%%%%%%%%

\section*{Acknowledgements} 

We thank the referee for many comments and suggestions.
%%%%%%%%%%%%%%%%%%%%%%%%%%%%%%%%%%%%%%%%%%%%%%%%%%%%%%%%%%%%

\end{document}